\documentclass[12pt,reqno]{amsart} 
\usepackage{amsthm}

\usepackage{amsmath,amssymb}
\usepackage{ascmac}
\usepackage{mathrsfs}

\usepackage[abbrev]{amsrefs}

\usepackage{graphicx,xcolor}

\usepackage{hyperref}

\hypersetup{
    colorlinks=true,     % 枠ではなく文字自体をリンクにする
    linkcolor=red,      % セクション番号などのリンク
    citecolor=green,      % 文献番号のリンク
    urlcolor=blue,       % URLリンク
    bookmarks=true,
    bookmarksnumbered=true,
    bookmarksdepth=tocdepth
}

\allowdisplaybreaks[1]
\theoremstyle{definition}
\newtheorem{dfn}{Definition}
\newtheorem{thm}{Theorem}[section]

\newtheorem{prop}[thm]{Proposition}
\newtheorem{lem}[thm]{Lemma}
\newtheorem{rem}{Remark}[section]

\newtheorem*{nota}{Notation}

\newcommand{\eqsp}[1]{{\begin{equation}\begin{split}#1\end{split}\end{
equation}}}

\makeatletter
\newcommand{\subscripts}[3]{%
  \@mathmeasure\z@\displaystyle{#2}%
  \global\setbox\@ne\vbox to\ht\z@{}\dp\@ne\dp\z@
  \setbox\tw@\box\@ne
  \@mathmeasure4\displaystyle{\copy\tw@_{#1}}%
  \@mathmeasure6\displaystyle{{#2}_{#3}}%
  \dimen@-\wd6 \advance\dimen@\wd4 \advance\dimen@\wd\z@
  \hbox to\dimen@{}\mathop{\kern-\dimen@\box4\box6}%
}
\makeatother

\begin{document}

\begin{center}\LARGE \bf
  Uniqueness of the dissipative SQG without time-continuity assumption
\end{center}
   
\footnote[0]
{


* Taiki Okazaki -- E-mail: okazaki.taiki.r5@dc.tohoku.ac.jp

}
\vskip5mm

\begin{center}
  {\large Taiki Okazaki*} 
  \vskip2mm
  {\large Mathematical Institute, Tohoku University}\\
  {\large Sendai 980-8578 Japan}
 \end{center}

\vskip5mm

\begin{center}
\begin{minipage}{135mm}
\footnotesize
{{\sc Abstract.} }
We consider the uniqueness of the solution of the dissipative surface quasi-geostrophic equation, 
without assuming time-continuity and smallness of the solutions. 
We show that the uniqueness holds in the scale-critical Lebesgue spaces and non-homogeneous Besov spaces. 
The proof is based on the energy method, 
inspired by the approach introduced by Lions and Masmoudi~\cite{Li_Ma_2001} 
in the study of uniqueness for the Navier-Stokes equations. 
A key ingredient of the argument is the justification of the energy inequality 
via the smoothing effect of the fractional heat semigroup together with an iteration scheme based on the structure of the integral equation. 
\end{minipage}
\end{center}

\section{Introduction}
In this paper, we consider the dissipative surface quasi-geostrophic equation. 
\begin{equation}
  \label{SQG}
  \begin{cases}
    \partial_t \theta + \Lambda^\alpha \theta + ({\bf u} \cdot \nabla) \theta = 0,
          &\quad  \   t>0,x\in{\mathbb{R}}^2, \\
    {\bf u} = \nabla ^\perp \Lambda^{-1}\theta,
          &\quad  \ t>0,x\in{\mathbb{R}}^2,\\
    \theta(0,x) = \theta_0(x),
          &\quad  \ x\in{\mathbb{R}}^2,\\
  \end{cases}
\end{equation}
where $\Lambda^{\alpha} = (-\Delta)^{\frac{\alpha}{2}}$, $1\leq \alpha\leq 2$ and $\nabla^{\perp} = (-\partial_2, \partial_1)$. 
The fractional Laplacian is defined as follows using the Fourier transform: 
\begin{equation*}
  \Lambda^\alpha \theta (x) = \mathcal{F}^{-1}[|\xi|^\alpha \widehat{\theta}(\xi)](x). 
\end{equation*} 
%%物理的背景%% 
The real-valued unknown function $\theta(t,x)$ denotes the potential temperature 
and ${\bf u}(t,x)$ denotes the velocity of a fluid. 
The surface quasi-geostrophic equation is derived from the quasi-geostrophic approximation of the three-dimensional Navier-Stokes equations 
and the assumption that the interior potential vorticity and the buoyancy frequency are constant. 
It is an evolution equation for the temperature at the boundary in the vertical direction 
and is known to be an important model in geophysical dynamics. 
For a detailed derivation of the equation and its physical background, we refer the reader to the papers~\cite{Co_2002, He_Pi_Ga_Sw_1995, Pe_1979}. 
From a mathematical point of view, the surface quasi-geostrophic equation has a similar structure to the three-dimensional Euler and Navier-Stokes equations (see~\cite{Co_Ma_Ta_1994, Co_Ma_Ta_1994_2}). 
Also, when $\alpha = 2$, the scale-critical spaces of \eqref{SQG} are the same as those of the two-dimensional incompressible Navier-Stokes equations. 

%%scaleについて%% 
We briefly discuss the scale-critical spaces of \eqref{SQG}. 
If $\theta$ is a solution of \eqref{SQG}, 
then for $\lambda>0$, 
\begin{equation*}
  \theta_{\lambda}(t,x) := \lambda^{\alpha-1}\theta(\lambda^{\alpha}t, \lambda x)
\end{equation*}
is also a solution of \eqref{SQG}. 
Namely, \eqref{SQG} is invariant under this scaling. 
Then, if a function space $X$ satisfies 
\begin{equation*}
  {\|\theta_{\lambda}(0,\cdot)\|}_{X} = {\|\theta(0,\cdot)\|}_{X}, 
\end{equation*}
then $X$ is called a scale-critical space of \eqref{SQG}. 
It is known that scale-critical spaces play an important role in the study of well-posedness for initial value problems of partial differential equations (see, e.g.,\cite{Ko_Ta_2001, Bo_Pa_2008}). 
When $1\leq \alpha \leq 2$, the scale-critical space in the framework of Lebesgue spaces is $L^{\frac{2}{\alpha-1}}(\mathbb{R}^{2})$. 
In particular, when $\alpha=1$, the scale-critical space is $L^{\infty}$. 
On the other hand, in the Besov space setting, the homogeneous Besov norm ${\|\cdot\|}_{\dot{B}_{p,q}^{s}(\mathbb{R}^{d})}$ satisfies the following scaling property: 
\begin{equation*}
  c\lambda^{s-\frac{d}{p}}{\|f\|}_{\dot{B}_{p,q}^{s}(\mathbb{R}^{d})} 
  \leq {\|f(\lambda \cdot)\|}_{\dot{B}_{p,q}^{s}(\mathbb{R}^{d})} 
  \leq C\lambda^{s-\frac{d}{p}}{\|f\|}_{\dot{B}_{p,q}^{s}(\mathbb{R}^{d})}. 
\end{equation*}
Thus, $\dot{B}_{p,q}^{1 - \alpha +\frac{2}{p}}(\mathbb{R}^{2})$ is a scale critical space of \eqref{SQG}. 
In this paper, we also refer to the non-homogeneous Besov space $B_{p,q}^{1 - \alpha +\frac{2}{p}}(\mathbb{R}^{2})$ as a scale-critical space of \eqref{SQG}. 

%%先行研究(well-posedness)%%
We begin by recalling several known results. 
Here, we consider $0< \alpha \leq 2$. 
From the balance between the nonlinearity and the dissipation, 
the cases $\alpha>1$, $\alpha = 1$, $\alpha<1$ are called the sub-critical, critical and super-critical cases respectively. 
In the sub-critical case, the global-in-time regularity for large initial data was obtained by Constantin and Wu~\cite{Co_Wu_1999}. 
Also, Carrillo and Ferreira~\cite{Ca_Fe_2008} proved the global well-posedness of \eqref{SQG} in the scale-critical Lebesgue spaces $L^{\frac{2}{\alpha-1}}(\mathbb{R}^2)$. 
In the critical case, Constantin, Cordoba and Wu~\cite{Co_Co_Wu_2001} proved the global existence and regularity for small data. 
Later, the large initial data case was solved by Kiselev, Nazarov and Volberg~\cite{Ki_Na_Vo_2007} and Caffarelli and Vasseur~\cite{Ca_Va_2010}. 
On the other hand, in the super-critical case, global-in-time regularity for large initial data still remains an open problem. 
For small initial data, global regularity is known (see e.g.~\cite{Co_Vi_2016}). 

%%先行研究(uniqueness)%%
Concerning the uniqueness of the solution, in the sub-critical case, 
the uniqueness in the scale-critical Lebesgue space $C([0,T]; L^{\frac{2}{\alpha-1}})$ was established by Ferreira~\cite{Fe_2011} ($1<\alpha<2$) and Iwabuchi and Ueda~\cite{I_U_2024} ($\alpha=2$). 
Recently, for $0<\alpha\leq 2$, Iwabuchi and Okazaki~\cite{I_O_2026} proved that the uniqueness of solutions holds in scale-critical non-homogeneous Besov spaces $C([0,T]; B_{p,q}^{1-\alpha+\frac{2}{p}})$ 
with $2/(\alpha-1)\leq p\leq \infty$, $1\leq q < \infty$ and $1-\alpha+2/p>-1/2$. 
Note that $-1/2$ is characterized as the least regularity for which the nonlinear term in \eqref{SQG} is well defined. 
As for non-uniqueness, if $0<\alpha<3/2$, then Buckmaster, Shkoller and Vicol~\cite{Bu_Sh_Vi_2019} proved the non-uniqueness of weak solutions in $L_{\text{loc}}^2(0,T; \dot{H}^{-\frac{1}{2}}(\mathbb{T}^2))$. 
In a related direction, for the forced surface quasi-geostrophic equation, Mengual and Solera~\cite{Me_So_2026} constructed a forcing term for which the non-uniqueness of solutions occurs in the scale-critical space $L^{\infty}(0, T; L^{\frac{2}{\alpha - 1}})$ ($1 < \alpha \leq 2$). 
Their result concerns the forced equation with a non-trivial forcing term, 
whereas in the present paper we consider the uniqueness problem for the unforced equation in the same function space. 

%%先行研究(Navier-Stokesのuniqueness)%%
We also refer to some literature on the Navier-Stokes equations. 
For the two-dimensional incompressible Navier-Stokes equations, 
it is a classical result that the weak solution ${\bf u}\in L^\infty(0,T; L^2)\cap L^2(0,T; \dot{H}^{1})$ is unique if ${\bf u}_0\in L^2$ (see, e.g.,~\cite{Te_1995} and the references therein). 
The uniqueness in the scale-critical Lebesgue space $C([0,T]; L^{2})$ does not seem to be resolved. 
Regarding the non-uniqueness result, Cheskidov and Luo~\cite{Ch_Lu_2023} established that the non-uniqueness holds in $C([0,T]; L^{p}(\mathbb{T}^{2}))$ with $1\leq p <2$. 
In the case when the space dimension is three, the uniqueness of the solution in the scale-critical space $C([0,T];L^{3})$ 
was proved by Meyer~\cite{Me_1997}, Furioli, Lemari\'{e}-Rieusset and Terraneo~\cite{Fu_Le_Te_1997} and Monniaux \cite{Mo_1999}. 
Also, Lions and Masmoudi~\cite{Li_Ma_2001} proved that the uniqueness holds in $L^{3}$ using a dual problem. 
The non-uniqueness in $C([0,T]; L^{2}(\mathbb{T}^{3}))$ was shown by Buckmaster and Vicol~\cite{Bu_Vi_2019}. 
Recently, if $d \geq 3$, then Fujii~\cite{Fu_2026} established that the non-uniqueness of the $d$-dimensional incompressible Navier-Stokes equations 
in scale-critical homogeneous Besov space $C([0,T]; B_{p,q}^{-1 + \frac{d}{p}}(\mathbb{R}^{d}))$, where $B_{p,q}^{-1 + \frac{d}{p}}(\mathbb{R}^{d})$ is strictly larger than $L^{d}(\mathbb{R}^{d})$. 
For solutions that are not continuous in time, 
if $d\geq 4$, then the uniqueness of the $d$-dimensional incompressible Navier-Stokes equations in $L^{\infty}(0,T; L^{d})$ is proved in~\cite{Li_Ma_2001}. 

%%今回の主定理%% 
In this work, in the sub-critical case, we prove that solutions of \eqref{SQG} are unique in scale-critical Lebesgue space $L^{\infty}(0,T; L^{\frac{2}{\alpha-1}})$ without smallness assumptions. 
In addition, we include the critical case $\alpha = 1$ and establish the uniqueness of solutions of \eqref{SQG} in the scale-critical non-homogeneous Besov spaces $B_{p,q}^{1 - \alpha + \frac{2}{p}}$. 

We first present the setting for the main result in the Lebesgue space setting. 
%%解の定義%% 
In this paper, the solution of \eqref{SQG} is defined as follows. 
\begin{dfn}\label{0215-1}
  Let $1<\alpha\leq 2$, $T>0$ and 
  $\theta_0\in L^{\frac{2}{\alpha-1}}$. 
  If $\theta:(0,T)\times\mathbb{R}^2\to\mathbb{R}$ satisfies 
  \begin{equation}\label{0203-1}
     \begin{cases}
      \theta \in L^\infty(0,T; L^{\frac{2}{\alpha-1}}), \\
      \subscripts{\mathcal{S}'}{\displaystyle\big\langle\theta(t), \phi \big\rangle}{\mathcal{S}}
      =\subscripts{\mathcal{S}'}{\big\langle e^{-t\Lambda^\alpha}\theta_0,\phi\big\rangle}{\mathcal{S}}
      +\subscripts{\mathcal{S}'}{\left\langle \int_{0}^{t}e^{-(t-s)\Lambda^\alpha}({\bf u}\theta) ~{\rm d}s,\nabla\phi \right\rangle}{\mathcal{S}}\\ 
      \text{for a.e.}\ t\in (0,T) \text{ and for any }\phi\in\mathcal{S}(\mathbb{R}^2),
     \end{cases}
  \end{equation}
  then we call $\theta$ a solution of \eqref{SQG}. 
\end{dfn}

\begin{rem}
  \begin{enumerate}
    \item The equality~\eqref{0203-1} means that $\theta$ satisfies the integral equation derived from \eqref{SQG} in the sense of distributions. 
    This is due to the fact that condition $\theta \in L^\infty(0,T; L^{\frac{2}{\alpha-1}})$ alone 
    is not sufficient to ensure that the nonlinear term in the integral equation is well-defined. 
    \item In fact, since the function $t\mapsto \subscripts{\mathcal{S}'}{\langle\theta(t), \phi \rangle}{\mathcal{S}}$ is continuous, 
    we can define the equation~\eqref{0203-1} for any $t\in [0,T]$. 
    Moreover, since $\mathcal{S}(\mathbb{R}^2)$ is dense in $L^{p}$ ($1\leq p < \infty$) and the smoothing effect of the fractional heat semigroup, 
    we can prove that if $\theta$ is a solution of \eqref{SQG}, then $\theta$ is weakly continuous in $L^{\frac{2}{\alpha-1}}$ on $[0,T]$. 
    Thus, we can define the initial value in the following sense, 
    \begin{equation*}
      \theta(t) \rightharpoonup \theta_0 \text{ in }L^{\frac{2}{\alpha-1}}, \text{ as }t\to 0. 
    \end{equation*}
  \end{enumerate}
\end{rem}

We next define the weak solution of \eqref{SQG}. 

\begin{dfn}\label{0204-2}
  Let $1<\alpha\leq 2$, $T>0$ and 
  $\theta_0\in L^{\frac{2}{\alpha-1}}$. 
  If $\theta:(0,T)\times\mathbb{R}^2\to\mathbb{R}$ satisfies 
  \begin{equation}\label{0213-1}
     \begin{cases}
      \theta \in L^\infty(0,T; L^{\frac{2}{\alpha-1}}), \\ 
      \displaystyle\int_{0}^{T}\hspace{-6pt}\Big(\hspace{-2pt}\subscripts{\mathcal{S}'}{\displaystyle\big\langle\theta(t), \partial_t \varphi(t) - \Lambda^{\alpha}\varphi(t) \big\rangle}{\mathcal{S}} 
      + \subscripts{\mathcal{S}'}{\displaystyle\big\langle{\bf u}(t)\theta(t), \nabla\varphi(t) \big\rangle}{\mathcal{S}}\hspace{-2pt}\Big){\rm d}t
      + \hspace{-2pt} \subscripts{\mathcal{S}'}{\displaystyle\big\langle\theta_0, \varphi(0) \big\rangle}{\mathcal{S}} \hspace{-3pt} =  0\\ 
      \text{for any }\varphi \in C^{\infty}([0,T]\times\mathbb{R}^2) 
      \text{ such that }\varphi(T)\equiv 0, \\ 
      \text{and for any }t\in [0,T], \varphi(t, \cdot)\in \mathcal{S}(\mathbb{R}^2), 
     \end{cases}
  \end{equation}
  then we call $\theta$ a weak solution of \eqref{SQG}. 
\end{dfn}

Actually, it is known that \eqref{0203-1} and \eqref{0213-1} are equivalent. 

%%2つの解の同値性%% 
\begin{prop}\label{0204-4}
  Let $1<\alpha \leq 2$, $T>0$, and $\theta \in L^\infty (0,T; L^{\frac{2}{\alpha-1}})$. 
  Then $\theta$ is a solution of the integral equation of \eqref{SQG} in the sense of Definition~\ref{0215-1} 
  if and only if $\theta$ is a weak solution of \eqref{SQG}. 
\end{prop}

Proposition~\ref{0204-4} can be proved using the same method as the corresponding equivalence proposition for the Navier-Stokes equations in Fabes, Jones and Rivi\'ere~\cite{Fa_Jo_Ri_1972} (see also~\cite{Li_Ma_2001}). 
For the convenience of the reader, we provide a proof of Proposition~\ref{0204-4} in Appendix~\ref{0213-2}. 

We state our main results in Lebesgue spaces. 
\vspace{-4pt}
%%主定理(Lebesgue)%% 
\begin{thm}\label{0204-1}
  Let $1<\alpha \leq 5/3$ and $T>0$. 
  Let $\theta^{(1)}, \theta^{(2)} \in L^{\infty}(0,T; L^{\frac{2}{\alpha-1}})$ be solutions of the integral equation of \eqref{SQG} in the sense of Definition~\ref{0215-1}
  with the same initial data $\theta_0 \in L^{\frac{2}{\alpha-1}}$. 
  Then $\theta^{(1)}(t) = \theta^{(2)}(t)$ in $L^{\frac{2}{\alpha-1}}$ for a.e. $t\in (0,T)$. 
\end{thm}
\vspace{-3pt}
As in the previous work~\cite{I_O_2026}, we can extend Theorem~\ref{0204-1} to Besov spaces with negative regularity index. 
To state this result, let us recall the definition of non-homogeneous Besov spaces. 
We refer to the book by Bahouri, Chemin and Danchin~\cite{Ba_Ch_Da_2011}. 
Let us first introduce Littlewood-Paley decomposition. 
\vspace{-2pt}
%%Littlewood-Paley decompositionの定義%% 
\begin{dfn}
  Let $\{\psi\}\cup\{\phi_j\}_{j\in\mathbb{Z}}\subset\mathcal{S}(\mathbb{R}^d)$ be such that 
  \begin{equation*}
    \text{supp }\widehat{\psi}\subset\{\xi\in\mathbb{R}^d\ | |\xi| \leq 4/3\}, 
  \end{equation*}
  \begin{equation*}
    \text{supp } \widehat{\phi}_{j}\subset\{\xi\in\mathbb{R}^d\ |\ 3/8 \cdot 2^{j-1}\leq |\xi| \leq 4/3 \cdot 2^{j-1}\} 
    \text{ for any }j\in\mathbb{Z}, 
  \end{equation*}
  \begin{equation*}
    \widehat{\psi}(\xi)+\sum_{j\in\mathbb{N}}\widehat{\phi_j}(\xi)=1 \text{ for any }\xi\in\mathbb{R}^d 
    \text{ and }
    \sum_{j\in\mathbb{Z}}\widehat{\phi_j}(\xi)=1 \text{ for any }\xi\in\mathbb{R}^d\backslash\{0\}. 
  \end{equation*}
\end{dfn}
\vspace{-2pt}
Based on this, we define the non-homogeneous Besov spaces. 
\vspace{-2pt}
%%non-homogenous Besovの定義%% 
\begin{dfn}
  Let $s\in\mathbb{R}$\ and $1\leq p,q \leq \infty$. We define the non-homogeneous Besov spaces as follows. 
  \begin{equation*}
    B^s_{p,q}=B^s_{p,q}(\mathbb{R}^d):=\{f\in \mathcal{S}'(\mathbb{R}^d)\ |\ {\|f\|}_{B^s_{p,q}}<\infty\}, 
  \end{equation*}
  where
  \begin{equation*}
    {\|f\|}_{B^s_{p,q}}:={\|\psi*f\|}_{L^p}+{\left\|\left\{2^{sj}{\|\phi_j*f\|}_{L^p}\right\}_{j\in\mathbb{N}}\right\|}_{l^q}. 
  \end{equation*}
\end{dfn}
\vspace{-2pt}
When attempting to define a solution of \eqref{SQG} in the sense of Definition~\ref{0215-1} in Besov spaces with negative regularity index, 
we need to ensure that the product $(\nabla^{\perp}\Lambda^{-1}\theta)\theta$ is well-defined. 
Indeed, it is known that if $s<0$, then there exist some $f,g \in B_{p,q}^{s}$ such that 
\begin{equation*}
  fg\notin \mathcal{S}'. 
\end{equation*}
At this point, we consider the sum of two products 
\begin{equation*}
  (\nabla^{\perp}\Lambda^{-1}f)g + (\nabla^{\perp}\Lambda^{-1}g)f. 
\end{equation*}
It can be written in terms of first-order derivatives (see~\cite{I_U_2024}). 
As a result, we can define $(\nabla^{\perp}\Lambda^{-1}f)g + (\nabla^{\perp}\Lambda^{-1}g)f$ in $\mathcal{S}'$ for any $f,g\in B_{p,q}^{s}$ 
with $-1/2<s<0$, $2\leq p \leq \infty$ and $1\leq q\leq \infty$ (see~\cite{I_O_2026}). 

Thus, we can define a solution of \eqref{SQG} (as well as a weak solution) in scale-critical non-homogeneous Besov spaces 
provided that the product $(\nabla^{\perp}\Lambda^{-1}\theta)\theta$ is well-defined. 
\vspace{-2pt}
%%Besov版の解のDef%%
\begin{dfn}
  Let $1\leq\alpha\leq 2$, $T>0$, $1\leq p,q \leq \infty$ with $1 - \alpha + 2/p > -1/2$ and 
  $\theta_0\in B_{p,q}^{1-\alpha + \frac{2}{p}}$. 
  If $\theta:(0,T)\times\mathbb{R}^2\to\mathbb{R}$ satisfies 
  \begin{equation*}
     \begin{cases}
      \theta \in L^\infty(0,T; B_{p,q}^{1-\alpha + \frac{2}{p}}), \\
      \subscripts{\mathcal{S}'}{\displaystyle\big\langle\theta(t), \phi \big\rangle}{\mathcal{S}}
      =\subscripts{\mathcal{S}'}{\big\langle e^{-t\Lambda^\alpha}\theta_0,\phi\big\rangle}{\mathcal{S}}
      +\subscripts{\mathcal{S}'}{\left\langle \int_{0}^{t}e^{-(t-s)\Lambda^\alpha}({\bf u}\theta) ~{\rm d}s,\nabla\phi \right\rangle}{\mathcal{S}}\\ 
      \text{for a.e.}\ t\in (0,T) \text{ and for any }\phi\in\mathcal{S}(\mathbb{R}^2),
     \end{cases}
  \end{equation*}
  then we call $\theta$ a solution of \eqref{SQG}. 
\end{dfn}

Then, in the sub-critical case, the following results hold. 

%%Besov版の主定理, end-point case%% 
\begin{thm}\label{0212-1}
  Let $1<\alpha \leq 5/3$ and $T>0$. 
  Let $\theta^{(1)}, \theta^{(2)} \in L^{\infty}(0,T; B_{\frac{4}{\alpha-1},2}^{-\frac{1}{2}(\alpha - 1)})$ be solutions of the integral equation of \eqref{SQG} in the sense of Definition~\ref{0215-1}
  with the same initial data $\theta_0 \in B_{\frac{4}{\alpha-1},2}^{-\frac{1}{2}(\alpha - 1)}$. 
  Then $\theta^{(1)}(t) = \theta^{(2)}(t)$ in $B_{\frac{4}{\alpha-1},2}^{-\frac{1}{2}(\alpha - 1)}$ for a.e. $t\in (0,T)$. 
\end{thm}

If $p< 4/(\alpha-1)$, then we can relax the assumption on the interpolation index $q$. 

%%Besov版の主定理, non-end-point case%% 
\begin{thm}\label{0212-8}
  Let $1<\alpha \leq 5/3$, $T>0$, $1\leq p \leq \infty$ with $2/(\alpha-1)\leq p <4/(\alpha-1)$ and $1\leq q <\infty$. 
  Let $\theta^{(1)}, \theta^{(2)} \in L^{\infty}(0,T; B_{p,q}^{1-\alpha + \frac{2}{p}})$ be solutions of the integral equation of \eqref{SQG} in the sense of Definition~\ref{0215-1}
  with the same initial data $\theta_0 \in B_{p,q}^{1-\alpha + \frac{2}{p}}$. 
  \begin{enumerate}
    \item If $\alpha=5/3$, 
    then $\theta^{(1)}(t) = \theta^{(2)}(t)$ in $B_{p,p}^{-\frac{2}{3} + \frac{2}{p}}$ for a.e. $t\in (0,T)$. 
    \item If $1<\alpha < 5/3$, 
    then $\theta^{(1)}(t) = \theta^{(2)}(t)$ in $B_{p,q}^{1-\alpha + \frac{2}{p}}$ for a.e. $t\in (0,T)$. 
  \end{enumerate}
\end{thm}

%%Lebesgueの場合の拡大になっていること%%
\begin{rem}
  For $p\geq 2/(\alpha - 1)$, it is known that 
  \begin{equation*}
    L^{\frac{2}{\alpha - 1}}(\mathbb{R}^{2}) \hookrightarrow B_{\frac{2}{\alpha - 1}, \frac{2}{\alpha - 1}}^{0}(\mathbb{R}^{2}) \hookrightarrow B_{p, \frac{2}{\alpha - 1}}^{1 - \alpha + \frac{2}{p}}(\mathbb{R}^{2}). 
  \end{equation*}
  That is, Theorem~\ref{0212-8} is an extension of Theorem~\ref{0204-1}. 
\end{rem}

In the critical case, the uniqueness of solutions of \eqref{SQG} holds in the scale-critical Besov spaces with finite integrability index and positive regularity index. 
%%$\alpha = 1$の場合について%% 
\begin{thm}\label{0409-1}
  Let $\alpha = 1$, $T>0$, $2\leq p < \infty$ and $1\leq q <\infty$. 
  Let $\theta^{(1)}, \theta^{(2)} \in L^{\infty}(0,T; B_{p,q}^{\frac{2}{p}})$ be solutions of the integral equation of \eqref{SQG} in the sense of Definition~\ref{0215-1}
  with the same initial data $\theta_0 \in B_{p,q}^{\frac{2}{p}}$. 
  Then $\theta^{(1)}(t) = \theta^{(2)}(t)$ in $B_{p,q}^{\frac{2}{p}}$ for a.e. $t\in (0,T)$. 
\end{thm}

%%手法の説明%%
We now present an outline of the proof of our main results. 
Our proof is motivated by the method employed in Lions and Masmoudi~\cite{Li_Ma_2001} to establish the uniqueness of the Navier-Stokes equations in $L^{\infty}(0,T; L^{d}(\mathbb{R}^{d}))$ ($d \geq 4$). 
In~\cite{Li_Ma_2001}, their approach relies on the energy equality that holds for the Navier-Stokes equations, 
\begin{equation*}
  \partial_t{\bf u} - \Delta{\bf u} + ({\bf u}\cdot \nabla){\bf u} + \nabla p = 0, \quad \nabla\cdot {\bf u}=0. 
\end{equation*}
If ${\bf u}, {\bf v}$ are solutions of the Navier-Stokes equations, then ${\bf u} - {\bf v}$ satisfies 
\begin{equation*}
  \partial_t({\bf u}-{\bf v}) - \Delta({\bf u}-{\bf v}) + ({\bf u}\cdot \nabla)({\bf u}-{\bf v}) + \big(({\bf u}-{\bf v})\cdot \nabla\big){\bf v} + \nabla p = 0. 
\end{equation*}
Thus, by the energy method, 
\begin{equation*}
  \frac{1}{2}\frac{\rm d}{{\rm d}t}{\|({\bf u}- {\bf v})(t)\|}_{L^2}^2 
  + {\|\nabla({\bf u}-{\bf v})(t)\|}_{L^2}^2 
  \leq -\int_{\mathbb{R}^{d}}({\bf u}- {\bf v})\big(({\bf u}-{\bf v})\cdot \nabla\big){\bf v}~{\rm d}x
\end{equation*}
holds, which implies ${\bf u}$ does not appear in the estimate. 
Hence, Lions and Masmoudi~\cite{Li_Ma_2001} assumed that ${\bf u}\in L^{\infty}(0,T; L^{d}(\mathbb{R}^{d}))$ and ${\bf v}\in C([0,T]; L^{d}(\mathbb{R}^{d}))$, and that ${\bf u}$, {\bf v} have same initial data. 
And they proved the uniqueness by using a decomposition for solutions that are continuous in time. 

The crucial point in the above argument is the justification of the energy inequality, 
that is, it is necessary to verify that 
\begin{equation}\label{0215-2}
  {\bf u} - {\bf v}\in L^{\infty}(0,T; L^{2})\cap L^2(0,T; \dot{H}^{1}) 
\end{equation}
holds. 
Indeed, \cite{Li_Ma_2001} established that~\eqref{0215-2} holds for ${\bf u}, {\bf v}\in L^{\infty}(0,T; L^{d}(\mathbb{R}^{d}))$ with $d\geq 4$ 
by using the smoothing effect of the Stokes operator combined with an iteration argument. 

We briefly explain why dimension constraints $d \geq 4$ arise. 
From the well-known Sobolev embedding, the following holds, 
\begin{equation*}
  H^{1}(\mathbb{R}^{d})\hookrightarrow W^{2-\frac{d}{2}, d}(\mathbb{R}^{d}), 
\end{equation*}
Thus, ${\bf u} - {\bf v}\in L^{2}(0,T; W^{2-\frac{d}{2}, d}(\mathbb{R}^{d}))$ is a necessary condition for \eqref{0215-2}. 
Note that if $d\geq 4$, then $2-d/2\leq 0$. 
On the other hand, we consider the following Duhamel term for the Navier-Stokes equations 
\begin{equation}\label{0215-3}
  \int_{0}^{t}e^{(t-s)\Delta}\mathbb{P}\big(\nabla\cdot ({\bf u}\otimes{\bf u})\big)~{\rm d}s, 
\end{equation}
where $\mathbb{P}$ is the Helmholtz projection. 
By the smoothing effect of the Stokes operator, we can show that if $d\geq4$, then for any ${\bf u}\in L^{\infty}(0,T; L^{d}(\mathbb{R}^{d}))$, 
we have 
\begin{equation*}
  \int_{0}^{t}e^{(t-s)\Delta}\mathbb{P}\big(\nabla\cdot ({\bf u}\otimes{\bf u})\big)~{\rm d}s 
  \in L^{2}(0,T; W^{2-\frac{d}{2}, d}(\mathbb{R}^{d})), 
\end{equation*}
which implies ${\bf u} - {\bf v}\in L^{2}(0,T; W^{2-\frac{d}{2}, d}(\mathbb{R}^{d}))$. 
In contrast, when $d = 2,3$, the Duhamel term \eqref{0215-3} belonging to $L^{2}(0,T; W^{2-\frac{d}{2}, d}(\mathbb{R}^{d}))$ is not generally guaranteed. 
From the above, it seems that the following condition 
\begin{equation}\label{0215-5}
  H^{1}(\mathbb{R}^{d})\hookrightarrow W^{2-\frac{d}{2}, d}(\mathbb{R}^{d}), \ 2-\frac{d}{2}\leq 0. 
\end{equation}
can be regarded as one of the criteria for determining whether the uniqueness can be established using the present method. 
In fact, if we consider the fractional Navier-Stokes equations 
\begin{equation*}
  \partial_t{\bf u} + \Lambda^{\alpha}{\bf u} + ({\bf u}\cdot \nabla){\bf u} + \nabla p = 0, \quad \nabla\cdot {\bf u}=0, 
\end{equation*}
then the scale-critical Lebesgue space is $L^{\frac{d}{\alpha-1}}(\mathbb{R}^{d})$, 
and the energy inequality 
\begin{equation*}
  \frac{1}{2}\frac{{\rm d}}{{\rm d}t}{\|{\bf u}(t)\|}_{L^2}^2 + {\|\Lambda^{\frac{\alpha}{2}}{\bf u}(t)\|}_{L^2}^2 \leq 0 
\end{equation*}
is justified under this condition. 
\begin{equation*}
  {\bf u} \in L^{\infty}(0,T; L^{2})\cap L^{2}(0,T; \dot{H}^{\frac{\alpha}{2}}). 
\end{equation*}
In this case, the condition corresponding to~\eqref{0215-5} is 
\begin{equation*}
  H^{\frac{\alpha}{2}}(\mathbb{R}^{d})\hookrightarrow W^{\frac{3}{2}\alpha - 1 - \frac{d}{2}, \frac{d}{\alpha-1}}(\mathbb{R}^{d}), \ \frac{3}{2}\alpha - 1-\frac{d}{2}\leq 0. 
\end{equation*}
Then, under the this condition, the uniqueness of solutions of fractional Navier-Stokes equations in $L^{\infty}(0,T; L^{\frac{d}{\alpha-1}}(\mathbb{R}^{d}))$ 
can be proved using a method similar to that of~\cite{Li_Ma_2001}. 
Regarding the range of $\alpha$, for instance, 
it is $\alpha\leq 4/3$ when $d = 2$, 
and $\alpha\leq 5/3$ when $d = 3$. 

In light of this, we discuss the energy equality in the surface quasi-geostrophic equation 
and the exponent $\alpha = 5/3$ in our theorem. 
Since ${\bf u} = \nabla^{\perp}\Lambda^{-1}\theta$ satisfies the incompressibility condition, 
the energy inequality also holds for \eqref{SQG}, 
\begin{equation}\label{0215-4}
  \frac{1}{2}\frac{{\rm d}}{{\rm d}t}{\|\theta(t)\|}_{L^2}^2 + {\|\Lambda^{\frac{\alpha}{2}}\theta(t)\|}_{L^2}^2 \leq 0. 
\end{equation}
In this case, the condition necessary to justify energy equality~\eqref{0215-4} is 
\begin{equation*}
  \theta \in L^{\infty}(0,T; L^{2})\cap L^{2}(0,T; \dot{H}^{\frac{\alpha}{2}}), 
\end{equation*}
this is the same as in the case of two-dimensional fractional Navier-Stokes equations. 
On the other hand, since $\nabla^{\perp}\cdot \nabla f=0$, we can express the nonlinear term in \eqref{SQG} as follows, 
\begin{equation*}
  ({\bf u}\cdot \nabla)\theta 
  = \left((\nabla^{\perp}\Lambda^{-1}\theta)\cdot \nabla\right)\theta 
  = \nabla^{\perp}\cdot \left((\Lambda^{-1}\theta)\nabla \theta\right). 
\end{equation*}
From this, the following inequality holds (see, e.g.,~\cite{Ma_2008}). 
\begin{equation}\label{0215-6}
  \frac{1}{2}\frac{{\rm d}}{{\rm d}t}{\|\Lambda^{-\frac{1}{2}}\theta(t)\|}_{L^2}^2 + {\|\Lambda^{\frac{\alpha}{2}-\frac{1}{2}}\theta(t)\|}_{L^2}^2 \leq 0. 
\end{equation}
Then the condition required to justify~\eqref{0215-6} is 
\begin{equation*}
  \theta \in L^{\infty}(0,T; \dot{H}^{-\frac{1}{2}})\cap L^{2}(0,T; \dot{H}^{\frac{\alpha}{2}-\frac{1}{2}}), 
\end{equation*}
and the condition corresponding to~\eqref{0215-5} is 
\begin{equation*}
  H^{\frac{\alpha}{2}-\frac{1}{2}}(\mathbb{R}^{2})\hookrightarrow W^{\frac{3}{2}\alpha - \frac{5}{2}, \frac{2}{\alpha-1}}(\mathbb{R}^{2}), \ \frac{3}{2}\alpha - \frac{5}{2}\leq 0
\end{equation*}
Clearly, 
\begin{equation*}
  \frac{3}{2}\alpha - \frac{5}{2}\leq 0 
  \text{ if and only if } 
  \alpha\leq 5/3, 
\end{equation*}
and the range of $\alpha$ is wider than that of two-dimensional fractional Navier-Stokes equations. 

The paper is organized as follows. 
In Section \ref{0327-1}, we recall several lemmas 
and justify the energy inequality \eqref{0215-6} for the solution defined in this paper. 
In Section \ref{0327-2}, we prove the main theorems using the energy method 
together with bilinear estimates exploiting the structure of the nonlinear term in \eqref{SQG}. 
Finally, in Appendix \ref{0213-2}, we show the equivalence of the two solutions defined in this paper by using the Littlewood-Paley decomposition. 

\begin{nota}
  Throughout this paper, $C>0$ and $c>0$ denote generic positive constants which may change from line to line. 
  In particular, when the constant $C$ depends on the time variable $T$, we write it as $C_{T}$. 
\end{nota}

\section{Preliminaries}\label{0327-1} 
In this section, we introduce several lemmas which are elemental properties of the fractional heat semigroup, homogeneous Sobolev spaces and homogeneous Besov spaces. 
Using these lemmas, we show that the difference of two solutions of \eqref{SQG} belongs to $L^{\infty}(0,T; \dot{H}^{-\frac{1}{2}})\cap L^{2}(0,T; \dot{H}^{\frac{\alpha}{2}-\frac{1}{2}})$. 

\subsection{Basic lemmas}
%%Fractional semigroupに対するL^p-L^q評価%% 
The following lemma gives an $L^{p} - L^{q}$ estimate for the fractional heat semigroup. 
\begin{lem}{\cite{Wu_2001}}\label{0206-9}
  Let $0 < \alpha \leq 2$, $t>0$ and $1\leq p\leq q \leq \infty$. 
  Then there exists a constant $C>0$ such that for any $f\in L^p(\mathbb{R}^2)$, we have 
  \begin{equation*}
    {\|\nabla e^{-t\Lambda^{\alpha}}f\|}_{L^q} 
    \leq C t^{-\left(\frac{1}{\alpha}+\frac{2}{\alpha}\left(\frac{1}{p} - \frac{1}{q}\right)\right)}{\|f\|}_{L^p}. 
  \end{equation*}
\end{lem}
\newpage 
%%Homogeneous Sobolevの定義%%
We define the homogeneous Sobolev spaces via Fourier multipliers (see~\cite{Tr_1983}). 
\begin{dfn}
  For $s\in \mathbb{R}$ and $1\leq p < \infty$, 
  we define the homogeneous Sobolev spaces as follows. 
  \begin{equation*}
    \dot{W}^{s,p} = \dot{W}^{s,p}(\mathbb{R}^{d}) := \{f\in \mathcal{S}'/ \mathcal{P} |\ {\|f\|}_{\dot{W}^{s,p}} < \infty\}, 
  \end{equation*}
  where $\mathcal{P}$ is the set of all polynomials of $d$ real variables and 
  \begin{equation*}
    {\|f\|}_{\dot{W}^{s,p}} := {\|\Lambda^{s}f\|}_{L^p}. 
  \end{equation*}
  When $p=2$, we denote $H^{s} := W^{s,2}$. 
\end{dfn}

%%$\dot{H}^{s}$の周波数分解を用いた同値ノルム%% 
We often use the equivalent norm of $\dot{H}^{s}$, which is given by the Littlewood-Paley decomposition. 
\begin{equation}\label{0406-1}
  c{\|f\|}_{\dot{H}^{s}} 
  \leq {\left\|\left\{2^{sj}{\left\|\phi_{j}*f\right\|}_{L^{2}}\right\}_{j\in\mathbb{Z}}\right\|}_{l^{2}} 
  \leq C{\|f\|}_{\dot{H}^{s}}. 
\end{equation}

The next lemma is the bilinear estimate in the homogeneous Sobolev spaces. 
%%斉次Sobolev空間における積の評価%% 
\begin{lem}[e.g., \cite{Ch_We_1991}]\label{0205-3}
  Let $0< s <1$ and $1<p, p_1, p_2, q_1, q_2<\infty$ with $1/p = 1/p_1 + 1/p_2 = 1/q_1 + 1/q_2$. 
  Then there exists a constant $C>0$ such that for any $f \in L^{p_1}\cap \dot{W}^{s, q_1}$ and $g \in\dot{W}^{s, p_2} \cap L^{q_2}$, we have 
  \begin{equation*}
    {\|fg\|}_{\dot{W}^{s, p}} 
    \leq C({\|f\|}_{L^{p_1}}{\|g\|}_{\dot{W}^{s, p_2}} + {\|f\|}_{\dot{W}^{s, q_1}}{\|g\|}_{L^{q_2}}). 
  \end{equation*}
\end{lem}

The following lemmas provide basic estimates for frequency-localized functions and fundamental properties of Besov spaces. 
%%周波数分解した関数の微分評価%% 
\begin{lem}[\cite{Ba_Ch_Da_2011}]\label{0205-7}
  Let $1\leq p \leq \infty$. 
  Then there exists a constant $C>0$ such that for any $j\in \mathbb{Z}$ and $f$ with $\phi_j*f \in L^p$, we have 
  \begin{equation}\label{0409-3}
    {\|\nabla(\phi_j*f)\|}_{L^p}\leq C 2^{j}{\|\phi_{j}*f\|}_{L^p} 
    \text{ and } 
    {\|\Lambda^{-1}(\phi_j*f)\|}\leq C 2^{-j}{\|\phi_{j}*f\|}_{L^p}. 
  \end{equation}
  Also, there exists a constant $C>0$ such that for any $f$ with $\psi*f \in L^{p}$, we have 
  \begin{equation}\label{0406-2}
    {\|\nabla(\psi*f)\|}_{L^p}\leq C {\|\psi*f\|}_{L^p}, 
  \end{equation}
\end{lem}

%%周波数分解した関数のsemigrop評価%%
\begin{lem}[\cite{Wu_Yu_2008}]\label{0206-2}
  Let $0<\alpha\leq 2$, $t>0$, $j\in\mathbb{Z}$ and $1\leq p\leq \infty$. 
  Then there exist constants $C,c>0$ such that for any $f$ with $\phi_j*f\in L^p$, we have 
  \begin{equation*}
    {\|e^{-t\Lambda^\alpha}(\phi_j*f)\|}_{L^p}\leq Ce^{-ct2^{\alpha j}}{\|\phi_j*f\|}_{L^p}. 
  \end{equation*}
\end{lem}

%%BesovのSobolev埋め込み%%
\begin{lem}[\cite{Ba_Ch_Da_2011}]\label{0205-6}
  Let $s\in\mathbb{R}$, $1\leq p_1\leq p_2\leq\infty$ and $1\leq q_1\leq q_2\leq\infty$. 
  Then $B_{p_1,q_1}^s(\mathbb{R}^d)\hookrightarrow B_{p_2,q_2}^{s-d(1/p_1-1/p_2)}(\mathbb{R}^d)$ 
  i.e., there exists a constant $C>0$ such that 
  for any $f\in B_{p_1,q_1}^s(\mathbb{R}^d)$, we have 
  \begin{equation*}
    {\|f\|}_{B_{p_2,q_2}^{s-d\left(\frac{1}{p_1} - \frac{1}{p_2}\right)}}\leq C{\|f\|}_{B_{p_1,q_1}^s}. 
  \end{equation*}
\end{lem}

%%BesovとSobolevの埋め込み(q = 1,\infty)%%
\begin{lem}[\cite{Ba_Ch_Da_2011}]\label{0205-4}
  Let $s\in \mathbb{R}$ and $1\leq p\leq \infty$. 
  Then $B_{p,1}^{s}\hookrightarrow W^{s,p}$ 
  i.e., there exists a constant $C>0$ such that for any $f\in B_{p,1}^{s}$, we have 
  \begin{equation}\label{0206-8}
    {\|f\|}_{W^{s,p}}\leq C {\|f\|}_{B_{p,1}^{s}}, 
  \end{equation}
  and $W^{s,p}\hookrightarrow B_{p,\infty}^{s}$
  i.e., there exist a constant $C>0$ such that 
  for any $f\in W^{s,p}$, we have 
  \begin{equation*}
    {\|f\|}_{B_{p,\infty}^{s}}\leq C{\|f\|}_{W^{s,p}}.
  \end{equation*}
\end{lem}

%%BesovとSobolevの埋め込み(q=2)%%
\begin{lem}[\cite{Ba_Ch_Da_2011}]\label{0205-8}
  Let $s\in\mathbb{R}$ and $1<p<\infty$. 
  If $1<p\leq 2$, then $W^{s,p}\hookrightarrow B_{p,2}^{s}$ i.e., there exists a constant $C>0$ such that for any $f\in W^{s,p}$, we have 
  \begin{equation}\label{0409-4}
    {\|f\|}_{B_{p,2}^{s}}\leq C {\|f\|}_{W^{s,p}}, 
  \end{equation}
  and if $2\leq p <\infty$, then $B_{p,2}^{s}\hookrightarrow W^{s,p}$ i.e., there exists a constant $C>0$ such that for any $f\in B_{p,2}^{s}$, we have 
  \begin{equation}\label{0409-5}
    {\|f\|}_{W^{s,p}}\leq C {\|f\|}_{B_{p,2}^{s}}. 
  \end{equation}
  In particular, $B_{2,2}^{s} = H^{s}$. 
\end{lem} 

%%Besovのduality%% 
\begin{lem}[\cite{Tr_1983}]\label{0212-5}
  Let $s\in \mathbb{R}$, $1\leq p, p', q, q'<\infty$ with $1 = 1/p + 1/p' = 1/q + 1/q'$. 
  Then the dual space of $B_{p,q}^{s}$ is identified with $B_{p',q'}^{-s}$. 
  Moreover, for any $f\in B_{p,q}^{s}$, we have 
  \begin{equation*}
    {\|f\|}_{B_{p,q}^{s}} 
    = \sup_{{\|g\|}_{B_{p',q'}^{-s}}=1}|\langle f,g \rangle|. 
  \end{equation*}
\end{lem}

%%Fractional semigroupに対するKozono--Ogawa--Taniuchi型の評価%%
The following lemma describes the smoothing effect of the fractional heat semigroup in the non-homogeneous Besov spaces. 
The case $\alpha = 2$ was established by Kozono, Ogawa and Taniuchi~\cite{Ko_Og_Ta_2003}. 
When $0<\alpha <2$, the proof is similar to that of Zhai~\cite{Zh_2010} in the homogeneous Besov spaces. 
\begin{lem}\label{0206-1}
  Let $0<\alpha\leq 2$, $t>0$, $s_1, s_2\in \mathbb{R}$ with $s_1\leq s_2$ and $1\leq p, q \leq \infty$. 
  Then there exists a constant $C>0$ such that for any $f\in B_{p,q}^{s_{1}}$, we have 
  \begin{equation*}
    {\|e^{-t\Lambda^{\alpha}}f\|}_{B_{p, q}^{s_2}} 
    \leq C (1 + t^{-\frac{1}{\alpha}(s_{2} - s_{1})}){\|f\|}_{B_{p, q}^{s_1}}. 
  \end{equation*}
\end{lem}

We recall several lemmas on bilinear estimates in non-homogeneous Besov spaces. 
%%Besovの積の評価(周波数離れているところ)%%
\begin{lem}[\cite{Ba_Ch_Da_2011}]\label{0205-5}
  Let $s\in \mathbb{R}$ and $1\leq p, q, p_1, p_2\leq \infty$ with $1/p=1/p_1 + 1/p_2$. 
  Then there exists a constant $C>0$ such that for any $f\in L^{p_1}$ and $g \in B_{p_2, q}^{s}$, we have 
  \begin{equation}\label{0212-6}%%基本系%%
    {\left\|\sum_{l\geq 2}(\psi*f)(\phi_{l}*g) + \sum_{k\leq l - 2}(\phi_k*f)(\phi_l*g)\right\|}_{B_{p,q}^{s}} 
    \leq C {\|f\|}_{L^{p_1}}{\|g\|}_{B_{p_2, q}^{s}}. 
  \end{equation}
  Also, let $\varepsilon>0$ and $1\leq q_1, q_2 \leq \infty$ with $1/q=1/q_1 + 1/q_2$. 
  Then there exists a constant $C>0$ such that for any $f\in B_{p_1, q_1}^{-\varepsilon}$ and $g \in B_{p_2, q_2}^{s}$, we have 
  \begin{equation}\label{0212-9}%%負のregularity%%
    {\left\|\sum_{l\geq 2}(\psi*f)(\phi_{l}*g) + \sum_{k\leq l - 2}(\phi_k*f)(\phi_l*g)\right\|}_{B_{p,q}^{s-\varepsilon}} 
    \leq C {\|f\|}_{B_{p_1, q_1}^{-\varepsilon}}{\|g\|}_{B_{p_2, q_2}^{s}}. 
  \end{equation}
\end{lem}

%%Besovの積の評価(remainder, 基本系)%% 
\begin{lem}\label{0212-7}
  Let $s_1, s_2\in \mathbb{R}$ with $s_1 + s_2>0$ and $1\leq p, q, p_1, p_2, q_1, q_2 \leq \infty$ with $1/p=1/p_1 + 1/p_2$ and $1/q=1/q_1 + 1/q_2$. 
  Then there exists a constant $C>0$ such that for any $f\in B_{p_1, q_1}^{s_1}$ and $g \in B_{p_2, q_2}^{s_2}$, we have 
  \begin{equation*}
    {\left\|\sum_{|k-l|\leq 1}(\phi_k*f)(\phi_l*g)\right\|}_{B_{p,q}^{s_1 + s_2}} 
    \leq C {\|f\|}_{B_{p_1, q_1}^{s_1}}{\|g\|}_{B_{p_2, q_2}^{s_2}}. 
  \end{equation*}
\end{lem}

The next lemma establishes a bilinear estimate in Besov spaces with negative regularity indices, 
based on the derivative structure of the nonlinear terms in \eqref{SQG}. 
%%Derivative structureを使ったリマインダーの評価%% 
\begin{lem}[\cite{I_O_2026}]\label{0205-9}
  Let $s>-2$, $s'\in \mathbb{R}$, $1\leq p, q, p_1, p_2, q_1, q_2 \leq \infty$ with $1/p = 1/p_1 + 1/p_2$ and $1/q = 1/q_1 + 1/q_2$. 
  Then there exists a constant $C>0$ such that for any $f \in B_{p_1, q_1}^{s'}$ and $g\in B_{p_2, q_2}^{s+1-s'}$, we have 
  \begin{align*}
    &{\left\|\sum_{|k-l|\leq 1}\Big(\big((\nabla^{\perp}\Lambda^{-1}(\phi_k*f))\cdot\nabla\big)(\phi_l* g) + \big((\nabla^{\perp}\Lambda^{-1}(\phi_l*g))\cdot\nabla\big)(\phi_k* f)\Big)\right\|}_{B_{p,q}^{s}}\\ 
    \leq & C {\|f\|}_{B_{p_1, q_1}^{s'}}{\|g\|}_{B_{p_2, q_2}^{s+1-s'}}. 
  \end{align*}
\end{lem}

At the end of this subsection, we describe the frequency decomposition of functions used in the proof of the main theorem. 
The following lemmas follow immediately from the definition of Besov spaces. 
%%High freq.での挙動%% 
\begin{lem}\label{0212-2}
  Let $s\in \mathbb{R}$, $1\leq p \leq \infty$ and $1\leq q <\infty$. 
  Then for any $\varepsilon>0$ and $f\in B_{p,q}^{s}$, there exists $N\in \mathbb{N}$ such that we have 
  \begin{equation*}
    {\left\|\sum_{|j|>N}\phi_j*f\right\|}_{B_{p,q}^{s}}\leq \varepsilon. 
  \end{equation*}
\end{lem}

%%Low freq. での挙動%% 
\begin{lem}\label{0212-3}
  Let $s\in \mathbb{R}$, $1\leq p \leq \infty$ and $1\leq q <\infty$. 
  Then for any $s'\in\mathbb{R}$, $N\in \mathbb{N}$ and $f\in B_{p,q}^{s}$, we have 
  \begin{equation*}
    {\left\|\sum_{|j|\leq N}\phi_j*f\right\|}_{B_{p,q}^{s+s'}} < \infty. 
  \end{equation*}
\end{lem}

By density, the following also holds in Lebesgue spaces. 
%%High freq.での挙動(Lebesgue)%% 
\begin{lem}\label{0205-1}
  Let $1\leq p < \infty$. 
  Then for any $\varepsilon>0$ and $f\in L^p$, there exists $N\in \mathbb{N}$ such that we have 
  \begin{equation*}
    {\left\|\sum_{|j|>N}\phi_j*f\right\|}_{L^{p}}\leq \varepsilon. 
  \end{equation*}
\end{lem}

%%Low freq. での挙動(Lebesgue)%% 
\begin{lem}\label{0205-2}
  Let $1\leq p < \infty$. 
  Then for any $s\in\mathbb{R}$, $N\in \mathbb{N}$ and $f\in L^p$, we have 
  \begin{equation*}
    {\left\|\sum_{|j|\leq N}\phi_j*f\right\|}_{\dot{W}^{s, p}} < \infty. 
  \end{equation*}
\end{lem}

\subsection{Justification of the energy inequality in Lebesgue spaces}
In this subsection, we established a crucial step toward Theorem~\ref{0204-1}. 
We justify the energy inequality for the difference between two solutions of the integral equation of \eqref{SQG} in the sense of Definition~\ref{0215-1} with the same initial data 
by estimating the Duhamel term and performing an iteration based on the structure of the integral equation. 

We prove the following proposition. 
%%Lebesgueの場合にsolutionがエネルギークラスに属すること%% 
\begin{prop}\label{0204-3}
  Let $1<\alpha\leq 5/3$ and $T>0$. 
  Let $\theta, \bar{\theta} \in L^{\infty}(0,T; L^{\frac{2}{\alpha-1}})$ be solutions of the integral equation of \eqref{SQG} in the sense of Definition~\ref{0215-1} with the same initial data $\theta_0 \in L^{\frac{2}{\alpha-1}}$. 
  Then we have 
  \begin{equation*}
    \theta - \bar{\theta} \in L^{\infty}(0,T; \dot{H}^{-\frac{1}{2}})\cap L^{2}(0,T; \dot{H}^{\frac{\alpha}{2}-\frac{1}{2}}). 
  \end{equation*}
\end{prop}

To prove the proposition, we need the following lemmas concerning estimates of Duhamel terms. 

%%iterationなしでダイレクトでできるケース%%
\begin{lem}\label{0206-3}
  Let $3/2\leq \alpha \leq 2$ and $T>0$. 
  Then there exists a constant $C_{T}>0$ such that for any $\theta \in L^\infty (0,T; L^{\frac{2}{\alpha-1}})$, we have 
  \begin{equation*}
    {\left\|\int_{0}^{t}e^{-(t-s)\Lambda^{\alpha}}\Big(\big((\nabla ^\perp \Lambda^{-1}\theta)\cdot\nabla\big)\theta\Big) ~{\rm d}s\right\|}_{L^\infty(0,T; \dot{H}^{-\frac{1}{2}})} 
    \leq C_{T}{\|\theta\|}_{L^\infty(0,T; L^{\frac{2}{\alpha-1}})}^2. 
  \end{equation*}
\end{lem}
\begin{proof}
  From $\nabla \cdot \nabla^{\perp}f=0$ and the definition of the norm of $\dot{H}^{s}$, we obtain 
  \begin{equation*}
    {\left\|e^{-(t-s)\Lambda^{\alpha}}\Big(\big((\nabla ^\perp \Lambda^{-1}\theta)\cdot\nabla\big)\theta\Big)\right\|}_{\dot{H}^{-\frac{1}{2}}} 
    = {\left\|\nabla\Lambda^{-\frac{1}{2}}e^{-(t-s)\Lambda^{\alpha}}\big((\nabla ^\perp \Lambda^{-1}\theta)\theta\big)\right\|}_{L^{2}}. 
  \end{equation*}
  Note that thanks to the Fourier multiplier theorem and Lemma~\ref{0205-7} \eqref{0406-2}, we get 
  \begin{equation*}
    {\left\|\nabla\Lambda^{-\frac{1}{2}} \psi * f\right\|}_{L^{2}} 
    \leq C {\left\|\psi * f\right\|}_{L^{2}}. 
  \end{equation*}
  Thus, since $B_{2,2}^{s} = H^{s}$ (Lemma~\ref{0205-8}) and using Lemma~\ref{0206-1}, we have 
  \begin{align}\label{0416-1}
    {\left\|e^{-(t-s)\Lambda^{\alpha}}\Big(\big((\nabla ^\perp \Lambda^{-1}\theta)\cdot\nabla\big)\theta\Big)\right\|}_{\dot{H}^{-\frac{1}{2}}} 
    \leq& C {\left\|e^{-(t-s)\Lambda^{\alpha}}\big((\nabla ^\perp \Lambda^{-1}\theta)\theta\big)\right\|}_{B_{2,2}^{\frac{1}{2}}}\notag\\ 
    \leq& C \left(1 + (t-s)^{-\left(2 - \frac{5}{2\alpha}\right)}\right){\|(\nabla ^\perp \Lambda^{-1}\theta)\theta\|}_{H^{3 - 2\alpha}}. 
  \end{align}
  By Sobolev embedding $L^{\frac{1}{\alpha-1}} \hookrightarrow H^{3 - 2\alpha}$ (note that $3/2\leq 1/(\alpha-1) \leq 2$), 
  H\"older's inequality 
  and boundedness of Riesz transform in $L^p$ ($1<p<\infty$), we get 
  \begin{equation*}
    {\|(\nabla ^\perp \Lambda^{-1}\theta)\theta\|}_{H^{3 - 2\alpha}} 
    \leq C {\|\theta\|}_{L^{\frac{2}{\alpha-1}}}^2. 
  \end{equation*}
  Therefore, using Young's inequality, we obtain 
  \begin{align*}
    &{\left\|\int_{0}^{t}e^{-(t-s)\Lambda^{\alpha}}\Big(\big((\nabla ^\perp \Lambda^{-1}\theta)\cdot\nabla\big)\theta\Big) ~{\rm d}s\right\|}_{L^\infty(0,T; \dot{H}^{-\frac{1}{2}})}\\ 
    \leq& C \int_{0}^{T}\left(1 + t^{-\left(2 - \frac{5}{2\alpha}\right)}\right)~{\rm d}t {\|\theta\|}_{L^\infty(0,T; L^{\frac{2}{\alpha-1}})}^2\\ 
    \leq& C_{T}{\|\theta\|}_{L^\infty(0,T; L^{\frac{2}{\alpha-1}})}^2. 
  \end{align*}
\end{proof}

Applying an argument similar to that in Lemma~\ref{0206-3}, we estimate Duhamel terms in $L^{2}(0,T; \dot{H}^{\frac{\alpha}{2}-\frac{1}{2}})$. 
Then the following exponent appears in the estimate corresponding to \eqref{0416-1}, 
\begin{equation*}
  -\frac{1}{\alpha}\left(\frac{\alpha}{2} + \frac{1}{2} - 3 + 2\alpha\right) 
  = -\frac{5}{2\alpha}\left(\alpha - 1\right). 
\end{equation*}
Clearly, $- 5(\alpha - 1)/2\alpha > -1$ if and only if $\alpha < 5/3$. 
That is, the following lemma holds. 

%%時間L^2, non end-point case%% 
\begin{lem}\label{0206-4}
  Let $T>0$ and $3/2\leq \alpha < 5/3$. 
  Then there exists a constant $C_{T}>0$ such that for any $\theta \in L^\infty (0,T; L^{\frac{2}{\alpha-1}})$, we have 
  \begin{equation*}
    {\left\|\int_{0}^{t}e^{-(t-s)\Lambda^{\alpha}}\Big(\big((\nabla ^\perp \Lambda^{-1}\theta)\cdot\nabla\big)\theta\Big) ~{\rm d}s\right\|}_{L^{2}(0,T; \dot{H}^{\frac{\alpha}{2}-\frac{1}{2}})} 
    \leq C_{T}{\|\theta\|}_{L^\infty(0,T; L^{\frac{2}{\alpha-1}})}^2. 
  \end{equation*}
\end{lem}

In case $\alpha = 5/3$, we need to work in Besov spaces. 
%%時間L^2, end-point case%%
\begin{lem}\label{0206-5}
  Let $T>0$. 
  Then there exists a constant $C_{T}>0$ such that for any $\theta \in L^\infty(0,T; L^3)$, we have 
  \begin{equation*}
    {\left\|\int_{0}^{t}e^{-(t-s)\Lambda^{\frac{5}{3}}}\Big(\big((\nabla ^\perp \Lambda^{-1}\theta)\cdot\nabla\big)\theta\Big) ~{\rm d}s\right\|}_{L^2(0,T; \dot{H}^{\frac{1}{3}})} 
    \leq C_{T}{\|\theta\|}_{L^\infty(0,T; L^{3})}^2. 
  \end{equation*}
\end{lem}
\begin{proof}
  By $H^{s}\hookrightarrow \dot{H}^{s}$ ($s>0$) and $H^{s} = B_{2,2}^{s}$, using Minkowski's inequality, we get 
  \begin{align*}
    &{\left\|\int_{0}^{t}e^{-(t-s)\Lambda^{\frac{5}{3}}}\Big(\big((\nabla ^\perp \Lambda^{-1}\theta)\cdot\nabla\big)\theta\Big) ~{\rm d}s\right\|}_{L^2(0,T; \dot{H}^{\frac{1}{3}})}\\ 
    \leq& C{\left\|\psi*\left(\int_{0}^{t}e^{-(t-s)\Lambda^{\frac{5}{3}}}\Big(\big((\nabla ^\perp \Lambda^{-1}\theta)\cdot\nabla\big)\theta\Big) ~{\rm d}s\right)\right\|}_{L^2(0,T;L^{2})}\\ 
    &+ C{\left\|\left\{2^{\frac{1}{3}j}{\left\|\phi_j*\left(\int_{0}^{t}e^{-(t-s)\Lambda^{\frac{5}{3}}}\Big(\big((\nabla ^\perp \Lambda^{-1}\theta)\cdot\nabla\big)\theta\Big) ~{\rm d}s\right)\right\|}_{L^2(0,T;L^{2})}\right\}_{j\in\mathbb{N}}\right\|}_{l^2}. 
  \end{align*}
  From the boundedness of $e^{-t\Lambda^{\frac{5}{3}}}$ in $L^{2}$ and using Young's inequality and Lemma~\ref{0205-6}, we have 
  \begin{equation*}
    {\left\|\psi*\left(\int_{0}^{t}e^{-(t-s)\Lambda^{\frac{5}{3}}}\Big(\big((\nabla ^\perp \Lambda^{-1}\theta)\cdot\nabla\big)\theta\Big) ~{\rm d}s\right)\right\|}_{L^2(0,T;L^{2})} 
    \leq C T{\big\|\psi*\big((\nabla ^\perp \Lambda^{-1}\theta)\theta\big)\big\|}_{L^{2}(0,T; L^{\frac{3}{2}})}. 
  \end{equation*}
  For each $j\in \mathbb{N}$, by Lemma~\ref{0206-2}, Sobolev embedding $W^{\frac{1}{3}, \frac{3}{2}}(\mathbb{R}^{2}) \hookrightarrow L^{2}(\mathbb{R}^{2})$ and Lemma~\ref{0205-7} \eqref{0409-3}, we obtain 
  \begin{align*}
    &2^{\frac{1}{3}j}{\left\|\phi_j*\left(\int_{0}^{t}e^{-(t-s)\Lambda^{\frac{5}{3}}}\Big(\big((\nabla ^\perp \Lambda^{-1}\theta)\cdot\nabla\big)\theta\Big) ~{\rm d}s\right)\right\|}_{L^2(0,T;L^2)}\\ 
    \leq& C 2^{\frac{5}{3}j}{\left\|\int_{0}^{t} e^{-c(t-s)2^{\frac{5}{3} j}}{\big\|\phi_j*\big((\nabla ^\perp \Lambda^{-1}\theta)\theta\big)\big\|}_{L^{\frac{3}{2}}} ~{\rm d}s\right\|}_{L^2(0,T)}. 
  \end{align*}
  Thanks to Young's inequality, we have 
  \begin{align*}
    &2^{\frac{5}{3}j}{\left\|\int_{0}^{t} e^{-c(t-s)2^{\frac{5}{3} j}}{\big\|\phi_j*\big((\nabla ^\perp \Lambda^{-1}\theta)\theta\big)\big\|}_{L^{\frac{3}{2}}} ~{\rm d}s\right\|}_{L^2(0,T)}\\ 
    \leq & 2^{\frac{5}{3}j}{\left\|e^{-ct2^{\frac{5}{3} j}}\right\|}_{L^{1}(0,T)}{\big\|\phi_j*\big((\nabla ^\perp \Lambda^{-1}\theta)\theta\big)\big\|}_{L^2(0,T; L^{\frac{3}{2}})}\\ 
    \leq& C{\big\|\phi_j*\big((\nabla ^\perp \Lambda^{-1}\theta)\theta\big)\big\|}_{L^{2}(0,T; L^{\frac{3}{2}})}. 
  \end{align*}
  Since $L^{\frac{3}{2}}\hookrightarrow B_{\frac{3}{2}, 2}^{0}$ (Lemma~\ref{0205-8} \eqref{0409-4}), we obtain 
  \begin{align*}
    {\left\|\int_{0}^{t}e^{-(t-s)\Lambda^{\frac{5}{3}}}\Big(\big((\nabla ^\perp \Lambda^{-1}\theta)\cdot\nabla\big)\theta\Big) ~{\rm d}s\right\|}_{L^{2}(0,T; \dot{H}^{\frac{1}{3}})} 
    &\leq C_{T}{\|(\nabla ^\perp \Lambda^{-1}\theta)\theta\|}_{L^{2}(0,T; L^{\frac{3}{2}})}\\ 
    &\leq C_{T}{\|\theta\|}_{L^\infty(0,T; L^{3})}^2. 
  \end{align*}
\end{proof}

The following lemma follows from $\nabla\cdot \nabla^{\perp}f = 0$ and Lemma~\ref{0206-9}. 

%%$1<\alpha< 3/2$, first iteration%% 
\begin{lem}\label{0206-10}
  Let $1<\alpha < 3/2$, $T>0$ and $\max\{1, 1/(\alpha-1)\} \leq p < 2/(\alpha-1)$.
  Then there exists a constant $C_{T}>0$ such that for any $\theta \in L^{\infty}(0,T ; L^{\frac{2}{\alpha-1}})$, we have 
  \begin{equation*}
    {\left\|\int_{0}^{t}e^{-(t-s)\Lambda^{\alpha}}\Big(\big((\nabla ^\perp \Lambda^{-1}\theta)\cdot\nabla\big)\theta\Big) ~{\rm d}s\right\|}_{L^\infty(0,T; L^{p})} 
    \leq C_{T}{\|\theta\|}_{L^\infty(0,T; L^{\frac{2}{\alpha-1}})}^2. 
  \end{equation*}
\end{lem}

Using the above lemmas, we prove Proposition~\ref{0204-3}. 
%%エネルギークラスに入ることの証明%% 
\begin{proof}[Proof of Proposition~\ref{0204-3}]
  Since $\theta, \bar{\theta} \in L^{\infty}(0,T; L^{\frac{2}{\alpha-1}})$ are solutions of the integral equation of \eqref{SQG} in the sense of Definition~\ref{0215-1}, we can write for any $\phi \in \mathcal{S}(\mathbb{R}^2)$, 
  \begin{align*}
    &\int_{\mathbb{R}^2}(\theta(t,y) - \bar{\theta}(t,y))\phi(y)~{\rm d}y\\ 
    =& \int_{\mathbb{R}^2}\left(\int_{0}^{t}e^{-(t-s)\Lambda^{\alpha}}\big((\nabla^{\perp}\Lambda^{-1}\theta)\theta - (\nabla^{\perp}\Lambda^{-1}\bar{\theta})\bar{\theta}\big)~{\rm d}s\right)\cdot\nabla \phi ~{\rm d}y\\ 
    =& \int_{\mathbb{R}^2}\left(\int_{0}^{t}e^{-(t-s)\Lambda^{\alpha}}\big((\nabla^{\perp}\Lambda^{-1}(\theta - \bar{\theta}))\theta + (\nabla^{\perp}\Lambda^{-1}\bar{\theta})(\theta - \bar{\theta})\big)~{\rm d}s\right)\cdot\nabla \phi ~{\rm d}y. 
  \end{align*}
  Fix $t\in (0,T)$, $x\in \mathbb{R}^2$ and $j\in \mathbb{Z}$. 
  For $\phi_j(x-\cdot)\in \mathcal{S}(\mathbb{R}^2)$, we obtain 
  \begin{equation*}
    \phi_j*(\theta - \bar{\theta})(t) 
    = -\nabla \phi_j * \left(\int_{0}^{t}e^{-(t-s)\Lambda^{\alpha}}\big((\nabla^{\perp}\Lambda^{-1}(\theta - \bar{\theta}))\theta + (\nabla^{\perp}\Lambda^{-1}\bar{\theta})(\theta - \bar{\theta})\big)~{\rm d}s\right). 
  \end{equation*}
  From the equivalent norm of $\dot{H}^{s}$ \eqref{0406-1}, we get 
  \begin{align*}
    &{\|\theta - \bar{\theta}\|}_{L^{\infty}(0,T; \dot{H}^{-\frac{1}{2}})}\\ 
    =& {\left\|\int_{0}^{t}e^{-(t-s)\Lambda^{\alpha}}\Big(\big((\nabla^{\perp}\Lambda^{-1}(\theta - \bar{\theta}))\cdot \nabla\big)\theta + \big((\nabla^{\perp}\Lambda^{-1}\bar{\theta})\cdot \nabla\big)(\theta - \bar{\theta})\Big)~{\rm d}s\right\|}_{L^{\infty}(0,T; \dot{H}^{-\frac{1}{2}})}, 
  \end{align*}
  and 
  \begin{align*}
    &{\|\theta - \bar{\theta}\|}_{L^{2}(0,T; \dot{H}^{\frac{\alpha}{2}-\frac{1}{2}})}\\ 
    =& {\left\|\int_{0}^{t}e^{-(t-s)\Lambda^{\alpha}}\Big(\big((\nabla^{\perp}\Lambda^{-1}(\theta - \bar{\theta}))\cdot \nabla\big)\theta + \big((\nabla^{\perp}\Lambda^{-1}\bar{\theta})\cdot \nabla\big)(\theta - \bar{\theta})\Big)~{\rm d}s\right\|}_{L^{2}(0,T; \dot{H}^{\frac{\alpha}{2}-\frac{1}{2}})}. 
  \end{align*}
  If $3/2\leq \alpha \leq 5/3$, then Proposition \ref{0204-3} follows directly from Lemma~\ref{0206-3}, Lemma~\ref{0206-4} and Lemma~\ref{0206-5}. 
  In the case of $1<\alpha <3/2$, we use an iteration scheme. 
  Thanks to Lemma~\ref{0206-10}, we have 
  \begin{equation*}
    \theta - \bar{\theta} \in L^\infty (0,T ; L^\frac{1}{\alpha-1}). 
  \end{equation*}
  On the other hand, since $e^{-t\Lambda^{\alpha}}\theta_0\in L^\infty (0,T ; L^{\frac{2}{\alpha-1}})$, we get 
  \begin{equation*}
    \theta, \bar{\theta} \in L^\infty (0,T ; L^{\frac{2}{\alpha-1}}) + L^\infty (0,T ; L^\frac{1}{\alpha-1}). 
  \end{equation*}
  Thus, by the same argument as in Lemma~\ref{0206-10}, we obtain 
  \begin{align*}
    &\int_{0}^{t}e^{-(t-s)\Lambda^{\alpha}}\Big(\big((\nabla^{\perp}\Lambda^{-1}(\theta - \bar{\theta}))\cdot \nabla\big)\theta + \big((\nabla^{\perp}\Lambda^{-1}\bar{\theta})\cdot \nabla\big)(\theta - \bar{\theta})\Big)~{\rm d}s\\ 
    \in& L^\infty (0,T ; L^{\frac{2}{3(\alpha-1)}}) + L^\infty (0,T ; L^\frac{1}{2(\alpha-1)}), 
  \end{align*}
  which implies $\theta - \bar{\theta} \in L^\infty (0,T ; L^{\frac{2}{3(\alpha-1)}}) + L^\infty (0,T ; L^\frac{1}{2(\alpha-1)})$. 
  Also, we have 
  \begin{equation*}
    \theta, \bar{\theta} \in L^\infty (0,T ; L^{\frac{2}{\alpha-1}}) 
    + L^\infty (0,T ; L^{\frac{1}{\alpha-1}}) 
    + L^\infty (0,T ; L^{\frac{2}{3(\alpha-1)}}) 
    + L^\infty (0,T ; L^{\frac{1}{2(\alpha-1)}}). 
  \end{equation*}
  Iterating this argument (and interpolation if necessary), we obtain 
  \begin{equation*}
    \theta - \bar{\theta} \in L^\infty(0,T; L^{\frac{2}{2-\alpha}}). 
  \end{equation*}
  Note that $2/(2-\alpha)> 2$ and 
  \begin{equation*}
    \frac{1}{2} = \frac{1}{\frac{2}{\alpha -1 }} + \frac{1}{\frac{2}{2 - \alpha}}. 
  \end{equation*}
  Hence, using Lemma~\ref{0206-1}, we have 
  \begin{align*}
    &{\left\|\int_{0}^{t}e^{-(t-s)\Lambda^{\alpha}}\Big(\big((\nabla^{\perp}\Lambda^{-1}(\theta - \bar{\theta}))\cdot \nabla\big)\theta + \big((\nabla^{\perp}\Lambda^{-1}\bar{\theta})\cdot \nabla\big)(\theta - \bar{\theta})\Big)~{\rm d}s\right\|}_{L^{\infty}(0,T; \dot{H}^{-\frac{1}{2}})}\\ 
    \leq& C {\left\|\int_{0}^{t}e^{-(t-s)\Lambda^{\alpha}}\Big(\big((\nabla^{\perp}\Lambda^{-1}(\theta - \bar{\theta}))\big)\theta + (\nabla^{\perp}\Lambda^{-1}\bar{\theta})(\theta - \bar{\theta})\Big)~{\rm d}s\right\|}_{L^{\infty}(0,T; B^{\frac{1}{2}}_{2,2})}\\ 
    \leq& C_{T}\big({\|\theta\|}_{L^\infty(0,T; L^{\frac{2}{\alpha - 1}})} + {\|\bar{\theta}\|}_{L^\infty(0,T; L^{\frac{2}{\alpha - 1}})}\big){\|\theta - \bar{\theta}\|}_{L^\infty(0,T; L^{\frac{2}{2-\alpha}})}, 
  \end{align*}
  and 
  \begin{align*}
    &{\left\|\int_{0}^{t}e^{-(t-s)\Lambda^{\alpha}}\Big(\big((\nabla^{\perp}\Lambda^{-1}(\theta - \bar{\theta}))\cdot \nabla\big)\theta + \big((\nabla^{\perp}\Lambda^{-1}\bar{\theta})\cdot \nabla\big)(\theta - \bar{\theta})\Big)~{\rm d}s\right\|}_{L^{2}(0,T; \dot{H}^{\frac{\alpha}{2}-\frac{1}{2}})}\\ 
    \leq& C_{T}\big({\|\theta\|}_{L^\infty(0,T; L^{\frac{2}{\alpha - 1}})} + {\|\bar{\theta}\|}_{L^\infty(0,T; L^{\frac{2}{\alpha - 1}})}\big){\|\theta - \bar{\theta}\|}_{L^\infty(0,T; L^{\frac{2}{2-\alpha}})}. 
  \end{align*}
\end{proof}

\subsection{Justification of the energy inequality in Besov spaces}
We next consider the case where the solution belongs to Besov spaces. 

Under the assumptions of Theorem~\ref{0212-1}, the following proposition holds. 

%%Besovの場合にエネルギークラスに入ること%% 
\begin{prop}\label{0211-1}
  Let $1<\alpha \leq 5/3$ and $T>0$. 
  Let $\theta, \bar{\theta} \in L^{\infty}(0,T; B_{\frac{4}{\alpha-1},2}^{-\frac{1}{2}(\alpha - 1)})$ be solutions of the integral equation of \eqref{SQG} in the sense of Definition~\ref{0215-1} with the same initial data $\theta_0 \in B_{\frac{4}{\alpha-1},2}^{-\frac{1}{2}(\alpha - 1)}$. 
  Then we have 
  \begin{equation*}
    \theta - \bar{\theta} \in L^{\infty}(0,T; \dot{H}^{-\frac{1}{2}})\cap L^{2}(0,T; \dot{H}^{\frac{1}{2} - \frac{\alpha}{2}}). 
  \end{equation*}
\end{prop}

We state lemmas needed to prove Proposition~\ref{0211-1}. 
%%first iteration%% 
\begin{lem}\label{0211-2}
  Let $1<\alpha \leq 5/3$ and $T>0$. 
  Then there exists a constant $C_{T}>0$ such that for any $\theta \in L^{\infty}(0,T; B_{\frac{4}{\alpha-1},2}^{-\frac{1}{2}(\alpha - 1)})$, we have 
  \begin{equation*}
    {\left\|\int_{0}^{t}e^{-(t-s)\Lambda^{\alpha}}\Big(\big((\nabla ^\perp \Lambda^{-1}\theta)\cdot\nabla\big)\theta\Big) ~{\rm d}s\right\|}_{L^{2}(0,T; L^{\frac{2}{\alpha-1}})} 
    \leq C_{T}{\|\theta\|}_{L^\infty(0,T; B_{\frac{4}{\alpha-1},2}^{-\frac{1}{2}(\alpha - 1)})}^2. 
  \end{equation*}
\end{lem}
\begin{proof}
  By $B_{\frac{2}{\alpha - 1},2}^0 \hookrightarrow L^{\frac{2}{\alpha-1}}$ (Lemma~\ref{0205-8} \eqref{0409-5}) and Minkowski's inequality, we have 
  \begin{align*}
    &{\left\|\int_{0}^{t}e^{-(t-s)\Lambda^{\alpha}}\Big(\big((\nabla ^\perp \Lambda^{-1}\theta)\cdot\nabla\big)\theta\Big) ~{\rm d}s\right\|}_{L^{2}(0,T; L^{\frac{2}{\alpha - 1}})}\\ 
    \leq& C{\left\|\psi*\left(\int_{0}^{t}e^{-(t-s)\Lambda^{\alpha}}\Big(\big((\nabla^{\perp}\Lambda^{-1}\theta)\cdot\nabla\big)\theta\Big)~{\rm d}s\right)\right\|}_{L^{2}(0,T; L^{\frac{2}{\alpha - 1}})}\\ 
    & + C{\left\|\left\{{\left\|\phi_j*\left(\int_{0}^{t}e^{-(t-s)\Lambda^{\alpha}}\Big(\big((\nabla^{\perp}\Lambda^{-1}\theta)\cdot\nabla\big)\theta\Big)~{\rm d}s\right)\right\|}_{L^{2}(0,T; L^{\frac{2}{\alpha - 1}})}\right\}_{j\in \mathbb{N}} \right\|}_{l^{2}}. 
  \end{align*}
  Since $e^{-(t-s)\Lambda^{\alpha}}$ is bounded on $L^{\frac{2}{\alpha-1}}$, we have 
  \begin{align*}
    &{\left\|\psi*\left(\int_{0}^{t}e^{-(t-s)\Lambda^{\alpha}}\Big(\big((\nabla^{\perp}\Lambda^{-1}\theta)\cdot\nabla\big)\theta\Big)~{\rm d}s\right)\right\|}_{L^{\frac{2}{\alpha - 1}}}\\ 
    \leq& C \int_{0}^{t}{\left\|\psi*\Big(\big((\nabla^{\perp}\Lambda^{-1}\theta)\cdot\nabla\big)\theta\Big)\right\|}_{L^{\frac{2}{\alpha - 1}}}~{\rm d}s. 
  \end{align*}
  For each $j\in \mathbb{N}$, using Lemma~\ref{0206-2}, we obtain 
  \begin{align*}
    &{\left\|\phi_j*\left(\int_{0}^{t}e^{-(t-s)\Lambda^{\alpha}}\Big(\big((\nabla^{\perp}\Lambda^{-1}\theta)\cdot\nabla\big)\theta\Big)~{\rm d}s\right)\right\|}_{L^{\frac{2}{\alpha - 1}}}\\ 
    \leq& C \int_{0}^{t}e^{-ct2^{\alpha}j}{\left\|\phi_j*\Big(\big((\nabla^{\perp}\Lambda^{-1}\theta)\cdot\nabla\big)\theta\Big)\right\|}_{L^{\frac{2}{\alpha - 1}}}~{\rm d}s. 
  \end{align*}
  Using Young's inequality, we get 
  \begin{equation*}
    {\left\|\int_{0}^{t}e^{-(t-s)\Lambda^{\alpha}}\Big(\big((\nabla ^\perp \Lambda^{-1}\theta)\cdot\nabla\big)\theta\Big) ~{\rm d}s\right\|}_{L^{2}(0,T; L^{\frac{2}{\alpha - 1}})} 
    \leq C_{T}{\left\|\big((\nabla^{\perp}\Lambda^{-1}\theta)\cdot\nabla\big)\theta\right\|}_{L^{2}(0,T; B_{\frac{2}{\alpha-1},2}^{-\alpha})}. 
  \end{equation*}
  By Bony's decomposition (see \cite{Bo_1981}), we have 
  \begin{align*}
    &{\left\|\big((\nabla^{\perp}\Lambda^{-1}\theta)\cdot\nabla\big)\theta\right\|}_{B_{\frac{2}{\alpha-1},2}^{-\alpha}}\\ 
    \leq& {\left\|\sum_{l\geq 2}\big((\nabla^{\perp}\Lambda^{-1}(\psi*\theta))\cdot \nabla\big)(\phi_l*\theta) + \sum_{k\leq l - 2}\big((\nabla^{\perp}\Lambda^{-1}(\phi_k*\theta))\cdot \nabla\big)(\phi_l*\theta)\right\|}_{B_{\frac{2}{\alpha-1},2}^{-\alpha}} \\ 
    &+ {\left\|\sum_{k\geq 2}\big((\nabla^{\perp}\Lambda^{-1}(\phi_{k}*\theta))\cdot \nabla\big)(\psi*\theta) + \sum_{l\leq k - 2}\big((\nabla^{\perp}\Lambda^{-1}(\phi_k*\theta))\cdot \nabla\big)(\phi_l*\theta)\right\|}_{B_{\frac{2}{\alpha-1},2}^{-\alpha}}\\ 
    &+ {\left\|\big((\nabla^{\perp}\Lambda^{-1}(\psi*\theta))\cdot \nabla\big)(\psi*\theta)\right\|}_{B_{\frac{2}{\alpha-1},2}^{-\alpha}}\\ 
    &+ {\left\|\big((\nabla^{\perp}\Lambda^{-1}(\psi*\theta))\cdot \nabla\big)(\phi_{1}*\theta) + \big((\nabla^{\perp}\Lambda^{-1}(\phi_{1}*\theta))\cdot \nabla\big)(\psi*\theta)\right\|}_{B_{\frac{2}{\alpha-1},2}^{-\alpha}}\\ 
    &+ {\left\|\sum_{|k-l|\leq 1}\big((\nabla^{\perp}\Lambda^{-1}(\phi_k*\theta))\cdot \nabla\big)(\phi_l*\theta)\right\|}_{B_{\frac{2}{\alpha-1},2}^{-\alpha}}. 
  \end{align*} 
  Using Lemma~\ref{0205-5} \eqref{0212-9} and Lemma~\ref{0205-7} \eqref{0409-3}, we obtain 
  \begin{align*}
    &{\left\|\sum_{l\geq 2}\big((\nabla^{\perp}\Lambda^{-1}(\psi*\theta))\cdot \nabla\big)(\phi_l*\theta) +\sum_{k\leq l - 2}\big((\nabla^{\perp}\Lambda^{-1}(\phi_k*\theta))\cdot \nabla\big)(\phi_l*\theta)\right\|}_{B_{\frac{2}{\alpha-1},2}^{-\alpha}}\\ 
    \leq& C {\|\theta\|}_{B_{\frac{4}{\alpha-1},2}^{-\frac{1}{2}(\alpha - 1)}}^2, 
  \end{align*}
  and 
  \begin{align*}
    &{\left\|\sum_{k\geq 2}\big((\nabla^{\perp}\Lambda^{-1}(\phi_{k}*\theta))\cdot \nabla\big)(\psi*\theta) +\sum_{l\leq k - 2}\big((\nabla^{\perp}\Lambda^{-1}(\phi_k*\theta))\cdot \nabla\big)(\phi_l*\theta)\right\|}_{B_{\frac{2}{\alpha-1},2}^{-\alpha}}\\ 
    \leq& C {\|\theta\|}_{B_{\frac{4}{\alpha-1},2}^{-\frac{1}{2}(\alpha - 1)}}^2. 
  \end{align*}
  Since $\psi$ and $\phi_{1}$ are supported around the origin, we have 
  \begin{equation*}
    {\left\|\big((\nabla^{\perp}\Lambda^{-1}(\psi*\theta))\cdot \nabla\big)(\psi*\theta)\right\|}_{B_{\frac{2}{\alpha-1},2}^{-\alpha}} 
    \leq C {\|\theta\|}_{B_{\frac{4}{\alpha-1},2}^{-\frac{1}{2}(\alpha - 1)}}^2, 
  \end{equation*}
  \begin{align*}
    &{\left\|\big((\nabla^{\perp}\Lambda^{-1}(\psi*\theta))\cdot \nabla\big)(\phi_{1}*\theta) + \big((\nabla^{\perp}\Lambda^{-1}(\phi_{1}*\theta))\cdot \nabla\big)(\psi*\theta)\right\|}_{B_{\frac{2}{\alpha-1},2}^{-\alpha}}\\ 
    \leq& C {\|\theta\|}_{B_{\frac{4}{\alpha-1},2}^{-\frac{1}{2}(\alpha - 1)}}^2. 
  \end{align*}
  Note that we can write 
  \begin{align*}
    &\sum_{|k-l|\leq 1}\big((\nabla^{\perp}\Lambda^{-1}(\phi_k*\theta))\cdot \nabla\big)(\phi_l*\theta)\\ 
    =& \frac{1}{2}\sum_{|k-l|\leq 1}\Big(\big((\nabla^{\perp}\Lambda^{-1}(\phi_k*\theta))\cdot \nabla\big)(\phi_l*\theta) 
     + \big((\nabla^{\perp}\Lambda^{-1}(\phi_l*\theta))\cdot \nabla\big)(\phi_k*\theta)\Big). 
  \end{align*}
  Thus, by Lemma~\ref{0205-9}, we get 
  \begin{equation*}
    {\left\|\sum_{|k-l|\leq 1}\big((\nabla^{\perp}\Lambda^{-1}(\phi_k*\theta))\cdot \nabla\big)(\phi_l*\theta)\right\|}_{B_{\frac{2}{\alpha-1},2}^{-\alpha}} 
    \leq C {\|\theta\|}_{B_{\frac{4}{\alpha-1},2}^{-\frac{1}{2}(\alpha - 1)}}^2. 
  \end{equation*}
  Therefore, we obtain 
  \begin{equation*}
    {\left\|\int_{0}^{t}e^{-(t-s)\Lambda^{\alpha}}\Big(\big((\nabla ^\perp \Lambda^{-1}\theta)\cdot\nabla\big)\theta\Big) ~{\rm d}s\right\|}_{L^{2}(0,T; L^{\frac{2}{\alpha-1}})} 
    \leq C_{T}{\|\theta\|}_{L^{\infty}(0,T; B_{\frac{4}{\alpha-1},2}^{-\frac{1}{2}(\alpha - 1)})}^2
  \end{equation*}
\end{proof}

Thanks to Lemma~\ref{0206-1} and using a similar product estimate to Lemma~\ref{0211-2}, we obtain following lemma. 
%%スケールを少しずらしたfirst iteration%%
\begin{lem}\label{0211-3}
  Let $1<\alpha\leq 5/3$, $T>0$ and $\varepsilon\in\mathbb{R}$ with $0<\varepsilon\leq \alpha$. 
  Then there exists a constant $C_{T}>0$ such that for any $\theta \in L^\infty(0,T; B_{\frac{4}{\alpha-1}, 2}^{-\frac{1}{2}(\alpha - 1)})$, we have 
  \begin{equation*}
    {\left\|\int_{0}^{t}e^{-(t-s)\Lambda^{\alpha}}\Big(\big((\nabla ^\perp \Lambda^{-1}\theta)\cdot\nabla\big)\theta\Big) ~{\rm d}s\right\|}_{L^{\infty}(0,T; B_{\frac{2}{\alpha-1}, 1}^{-\varepsilon})} 
    \leq C_{T}{\|\theta\|}_{L^\infty(0,T; B_{\frac{4}{\alpha-1}, 2}^{-\frac{1}{2}(\alpha - 1)})}^2. 
  \end{equation*}
\end{lem}

We now prove Proposition~\ref{0211-1}. 
\begin{proof}[Proof of Proposition~\ref{0211-1}]
  By a similar argument to Proposition~\ref{0204-3}, we estimate 
  \begin{equation*}
    {\left\|\int_{0}^{t}e^{-(t-s)\Lambda^{\alpha}}\Big(\big((\nabla^{\perp}\Lambda^{-1}(\theta - \bar{\theta}))\cdot \nabla\big)\theta + \big((\nabla^{\perp}\Lambda^{-1}\bar{\theta})\cdot \nabla\big)(\theta - \bar{\theta})\Big)~{\rm d}s\right\|}_{L^{\infty}(0,T; \dot{H}^{-\frac{1}{2}})}, 
  \end{equation*}
  and 
  \begin{equation*}
    {\left\|\int_{0}^{t}e^{-(t-s)\Lambda^{\alpha}}\Big(\big((\nabla^{\perp}\Lambda^{-1}(\theta - \bar{\theta}))\cdot \nabla\big)\theta + \big((\nabla^{\perp}\Lambda^{-1}\bar{\theta})\cdot \nabla\big)(\theta - \bar{\theta})\Big)~{\rm d}s\right\|}_{L^{2}(0,T; \dot{H}^{\frac{\alpha}{2} - \frac{1}{2}})}. 
  \end{equation*} 
  Note that we can write 
  \begin{equation*}
    {\|f\|}_{\dot{H}^{-\frac{1}{2}}} 
    = {\|\Lambda^{-\frac{1}{2}}f\|}_{L^{2}} 
    = {\|\Lambda^{\frac{1}{2}}\Lambda^{-1}f\|}_{L^{2}}. 
  \end{equation*}
  By $H^{s}\hookrightarrow \dot{H}^{s}$ $(s>0)$, $H^{s} = B_{2,2}^{s}$ and Lemma~\ref{0206-1}, we have for any $-\alpha/2 <  \varepsilon < \alpha/2$, 
  \begin{align*}
    &{\left\|\int_{0}^{t}e^{-(t-s)\Lambda^{\alpha}}\Big(\big((\nabla^{\perp}\Lambda^{-1}(\theta - \bar{\theta}))\cdot \nabla\big)\theta + \big((\nabla^{\perp}\Lambda^{-1}\bar{\theta})\cdot \nabla\big)(\theta - \bar{\theta})\Big)~{\rm d}s\right\|}_{L^{\infty}(0,T; \dot{H}^{-\frac{1}{2}})}\\ 
    \leq& C{\left\|\int_{0}^{t}e^{-(t-s)\Lambda^{\alpha}}\Lambda^{-1}\Big(\big((\nabla^{\perp}\Lambda^{-1}(\theta - \bar{\theta}))\cdot \nabla\big)\theta + \big((\nabla^{\perp}\Lambda^{-1}\bar{\theta})\cdot \nabla\big)(\theta - \bar{\theta})\Big)~{\rm d}s\right\|}_{L^{\infty}(0,T; B_{2,2}^{\frac{1}{2}})}\\ 
    \leq& C_{T}{\left\|\Lambda^{-1}\Big(\big((\nabla^{\perp}\Lambda^{-1}(\theta - \bar{\theta}))\cdot \nabla\big)\theta + \big((\nabla^{\perp}\Lambda^{-1}\bar{\theta})\cdot \nabla\big)(\theta - \bar{\theta})\Big)\right\|}_{L^{\infty}(0,T; B_{2,2}^{\frac{1}{2} - \frac{\alpha}{2}-\varepsilon})}. 
  \end{align*}
  Note that $\Lambda^{-1}$ has a singularity at the origin in the frequency space, 
  however, since $\nabla\cdot \nabla^{\perp} = 0$, we have 
  \begin{align*}
    {\left\|\Lambda^{-1}\left(\psi*\big(\big((\nabla^{\perp}\Lambda^{-1}f)\cdot\nabla\big)g\big)\right)\right\|}_{L^{2}} 
    =& {\left\|\nabla\Lambda^{-1}\left(\psi*\big((\nabla^{\perp}\Lambda^{-1}f)g\big)\right)\right\|}_{L^{2}}\\ 
    \leq& C {\left\|\psi*\big((\nabla^{\perp}\Lambda^{-1}f)g\big)\right\|}_{L^{2}}. 
  \end{align*}
  Also, using $H^{s}\hookrightarrow \dot{H}^{s}$ $(s>0)$, Minkowski's inequality and Lemma~\ref{0206-2}, we obtain 
  \begin{align*}
    &{\left\|\int_{0}^{t}e^{-(t-s)\Lambda^{\alpha}}\Big(\big((\nabla^{\perp}\Lambda^{-1}(\theta - \bar{\theta}))\cdot \nabla\big)\theta + \big((\nabla^{\perp}\Lambda^{-1}\bar{\theta})\cdot \nabla\big)(\theta - \bar{\theta})\Big)~{\rm d}s\right\|}_{L^{2}(0,T; \dot{H}^{\frac{\alpha}{2} - \frac{1}{2}})}\\ 
    \leq& C_{T}{\left\|\Big(\big((\nabla^{\perp}\Lambda^{-1}(\theta - \bar{\theta}))\cdot \nabla\big)\theta + \big((\nabla^{\perp}\Lambda^{-1}\bar{\theta})\cdot \nabla\big)(\theta - \bar{\theta})\Big)\right\|}_{L^{2}(0,T; H^{- \frac{\alpha}{2}- \frac{1}{2}})}. 
  \end{align*}
  Note that we can write 
  \begin{align*}
    &\big((\nabla^{\perp}\Lambda^{-1}(\theta - \bar{\theta}))\cdot \nabla\big)\theta + \big((\nabla^{\perp}\Lambda^{-1}\bar{\theta})\cdot \nabla\big)(\theta - \bar{\theta})\\ 
    =& \frac{1}{2}\Big(\big((\nabla^{\perp}\Lambda^{-1}(\theta - \bar{\theta}))\cdot \nabla\big)\theta + \big((\nabla^{\perp}\Lambda^{-1}\theta)\cdot \nabla\big)(\theta - \bar{\theta})\\ 
     &+ \big((\nabla^{\perp}\Lambda^{-1}(\theta - \bar{\theta}))\cdot \nabla\big)\bar{\theta} + \big((\nabla^{\perp}\Lambda^{-1}\bar{\theta})\cdot \nabla\big)(\theta - \bar{\theta})\Big). 
  \end{align*}
  We fix $\varepsilon$ to satisfy $-1/2 - \alpha/2 - \varepsilon > -2$. 
  We show that 
  \begin{equation}\label{0406-3}
    \theta - \bar{\theta}\in L^{\infty}(0,T; B_{\frac{4}{3-\alpha}, 2}^{-\varepsilon}) \cap L^2(0,T; L^{\frac{4}{3-\alpha}}). 
  \end{equation}
  Here 
  \begin{equation*}
    \frac{1}{2} = \frac{1}{\frac{4}{\alpha - 1}} + \frac{1}{\frac{4}{3 - \alpha}}. 
  \end{equation*}
  If \eqref{0406-3} holds, 
  then using product estimate (Lemma~\ref{0205-5} and Lemma~\ref{0205-9}), we obtain 
  \begin{align*}
    &{\left\|\Lambda^{-1}\Big(\big((\nabla^{\perp}\Lambda^{-1}(\theta - \bar{\theta}))\cdot \nabla\big)\theta + \big((\nabla^{\perp}\Lambda^{-1}\bar{\theta})\cdot \nabla\big)(\theta - \bar{\theta})\Big)\right\|}_{L^{\infty}(0,T; B_{2,2}^{\frac{1}{2} - \frac{\alpha}{2}-\varepsilon})}. \\ 
    \leq& C ({\|\theta\|}_{L^\infty(0,T; B_{\frac{4}{\alpha-1},2}^{-\frac{1}{2}\left(\alpha - 1\right)})} + {\|\bar{\theta}\|}_{L^\infty(0,T; B_{\frac{4}{\alpha-1},2}^{-\frac{1}{2}\left(\alpha - 1\right)})})
    {\|\theta - \bar{\theta}\|}_{L^{\infty}(0,T; B_{\frac{4}{3 - \alpha}, 2}^{-\varepsilon})}, 
  \end{align*}
  and 
  \begin{align*}
    &{\left\|\Big(\big((\nabla^{\perp}\Lambda^{-1}(\theta - \bar{\theta}))\cdot \nabla\big)\theta + \big((\nabla^{\perp}\Lambda^{-1}\bar{\theta})\cdot \nabla\big)(\theta - \bar{\theta})\Big)\right\|}_{L^{2}(0,T; H^{- \frac{\alpha}{2}- \frac{1}{2}})}\\ 
    \leq& C ({\|\theta\|}_{L^\infty(0,T; B_{\frac{4}{\alpha-1},2}^{-\frac{1}{2}\left(\alpha - 1\right)})} + {\|\bar{\theta}\|}_{L^\infty(0,T; B_{\frac{4}{\alpha-1},2}^{-\frac{1}{2}\left(\alpha - 1\right)})})
    {\|\theta - \bar{\theta}\|}_{L^2(0,T; L^{\frac{4}{3-\alpha}})}. 
  \end{align*}
  If $\alpha=\frac{5}{3}$, then $2/(\alpha-1) = 4/(3-\alpha) = 3$. 
  Thus, \eqref{0406-3} follows directly from Lemma~\ref{0211-2} and Lemma~\ref{0211-3}. 
  Let $1<\alpha<5/3$. 
  By Lemma~\ref{0211-2} and Lemma~\ref{0211-3}, we get 
  \begin{equation*}
    \theta-\bar{\theta}\in L^{\infty}(0,T; B_{\frac{2}{\alpha-1}, 1}^{-\varepsilon_1})\cap L^{2}(0,T; L^{\frac{2}{\alpha-1}}), 
  \end{equation*}
  and 
  \begin{equation*}
    \theta, \bar{\theta}
    \in L^{\infty}(0,T; B_{\frac{4}{\alpha-1}, 2}^{-\frac{1}{2}(\alpha-1)}) 
    + L^{\infty}(0,T; B_{\frac{2}{\alpha-1}, 1}^{-\varepsilon_1})\cap L^{2}(0,T; L^{\frac{2}{\alpha-1}}), 
  \end{equation*}
  where $0<\varepsilon_1\leq \alpha$. 
  We choose $\varepsilon_{1}$ sufficiently small so that $\varepsilon_1<1/2$. 
  Then, thanks to $-2\varepsilon_1-1>-2$, using a similar estimate to Lemma~\ref{0211-3}, we obtain 
  \begin{align*}
    &\int_{0}^{t}e^{-(t-s)\Lambda^{\alpha}}\Big(\big((\nabla^{\perp}\Lambda^{-1}(\theta - \bar{\theta}))\cdot \nabla\big)\theta + \big((\nabla^{\perp}\Lambda^{-1}\bar{\theta})\cdot \nabla\big)(\theta - \bar{\theta})\Big)~{\rm d}s\\ 
    \in& L^\infty (0,T ; B_{\frac{4}{3\alpha-3}, 1}^{-\varepsilon_2})\cap L^{2}(0,T ; L^{\frac{4}{3\alpha -3}})
     + L^\infty (0,T ; B_{\frac{1}{\alpha-1}, 1}^{-\varepsilon_3})\cap L^{2}(0,T; L^{\frac{1}{\alpha-1}}) , 
  \end{align*}
  where 
  \begin{equation*}
    \varepsilon_1-\frac{\alpha}{2}+\frac{1}{2} 
    < \varepsilon_2 
    < \varepsilon_1 + \frac{\alpha}{2} + \frac{1}{2} \quad\text{and}\quad 
    2\varepsilon_1 + 1 - \alpha 
    < \varepsilon_3 
    \leq 2\varepsilon_1 + 1. 
  \end{equation*}
  Note that we may take $\varepsilon_2$ and $\varepsilon_3$ to be negative. 
  Moreover, we can write 
  \begin{align*}
    &\big((\nabla^{\perp}\Lambda^{-1}\theta)\cdot\nabla\big)\theta\\ 
    =& \big((\nabla^{\perp}\Lambda^{-1}e^{-t\Lambda^{\alpha}}\theta)\cdot\nabla\big)e^{-t\Lambda^{\alpha}}\theta\\ 
     &+ \big((\nabla^{\perp}\Lambda^{-1}e^{-t\Lambda^{\alpha}}\theta)\cdot\nabla\big)(\theta - e^{-t\Lambda^{\alpha}}\theta )
     + \big((\nabla^{\perp}\Lambda^{-1}(\theta - e^{-t\Lambda^{\alpha}}\theta ))\cdot\nabla\big)e^{-t\Lambda^{\alpha}}\theta\\ 
     &+ \big((\nabla^{\perp}\Lambda^{-1}(\theta - e^{-t\Lambda^{\alpha}}\theta ))\cdot\nabla\big)(\theta - e^{-t\Lambda^{\alpha}}\theta), 
  \end{align*}
  which implies we can also apply Lemma~\ref{0205-9} to $\theta$ and $\bar{\theta}$. 
  Therefore, we may iterate this argument and we obtain 
  \begin{equation*}
    \theta - \bar{\theta}\in L^{\infty}(0,T; B_{\frac{3}{4 - \alpha}, 2}^{-\varepsilon}) \cap L^2(0,T; L^{\frac{4}{3-\alpha}}). 
  \end{equation*}
\end{proof}

The following Proposition will be used in the proof of Theorem~\ref{0212-8}. 
%%pがend-pointでないときに補間指数の制限を緩めてエネルギークラスに属すること%% 
\begin{prop}\label{0211-4}
  Let $1<\alpha\leq 5/3$, $T>0$ and $1\leq p\leq \infty$ with $2/(\alpha-1) \leq p <4/(\alpha-1)$. 
  \begin{enumerate}
    \item If $\alpha=5/3$ and 
    $\theta, \bar{\theta}\in L^{\infty}(0,T; B_{p, p}^{- \frac{2}{3} + \frac{2}{p}})$ are solutions of the integral equation of \eqref{SQG} in the sense of Definition~\ref{0215-1} with the same initial data $\theta_0\in B_{p, p}^{- \frac{2}{3} + \frac{2}{p}}$, 
    then we have 
    \begin{equation*}
      \theta - \bar{\theta} \in L^{\infty}(0,T; \dot{H}^{-\frac{1}{2}})\cap L^{2}(0,T; \dot{H}^{\frac{1}{3}}). 
    \end{equation*}
    \item \label{0417-2}If $1<\alpha<5/3$ and 
    $\theta, \bar{\theta}\in L^{\infty}(0,T; B_{p, \infty}^{1 - \alpha + \frac{2}{p}})$ are solutions of the integral equation of \eqref{SQG} in the sense of Definition~\ref{0215-1} with the same initial data $\theta_0\in B_{p, \infty}^{1 - \alpha + \frac{2}{p}}$, 
    then we have 
    \begin{equation*}
      \theta - \bar{\theta} \in L^{\infty}(0,T; \dot{H}^{-\frac{1}{2}})\cap L^{2}(0,T; \dot{H}^{\frac{\alpha}{2} - \frac{1}{2}}). 
    \end{equation*}
  \end{enumerate}
\end{prop}
Proposition~\ref{0211-4} can be proved by a similar argument to Proposition~\ref{0211-1}. 
We state the lemmas required for the proof and omit the proof of Proposition~\ref{0211-4}. 

%%$\alpha=5/3&のときの補完指数も含めた評価%%
The following lemma can be proved by the same method as in Lemma~\ref{0211-2}. 
\begin{lem}
  Let $T>0$ and $3\leq p \leq 6$. 
  Then there exists a constant $C_{T}>0$ such that for any $\theta \in L^{\infty}(0,T; B_{p,p}^{-\frac{2}{3}+\frac{2}{p}})$, we have 
  \begin{equation*}
    {\left\|\int_{0}^{t}e^{-(t-s)\Lambda^{\frac{5}{3}}}\Big(\big((\nabla ^\perp \Lambda^{-1}\theta)\cdot\nabla\big)\theta\Big) ~{\rm d}s\right\|}_{L^{p}(0,T; B_{\frac{p}{2}, \frac{p}{2}}^{-\frac{2}{3} + \frac{4}{p}})} 
    \leq C_{T}{\|\theta\|}_{L^\infty(0,T; B_{p,p}^{-\frac{2}{3}+\frac{2}{p}})}^2. 
  \end{equation*}
\end{lem}

%%$1<\alpha<5/3$のときに補完指数が自由に取れること%% 
In the proof of Proposition~\ref{0211-4}~\eqref{0417-2}, since the interpolation index is infinite, we use the next lemma instead of Lemma~\ref{0211-3}. 
\begin{lem}\label{0410-1}
  Let $1<\alpha<5/3$, $T>0$, $1\leq p \leq \infty$ with $2/(\alpha-1)< p < 4/(\alpha-1)$ and $s\in\mathbb{R}$ with $ s<1-\alpha+4/p$. 
  Then there exists a constant $C_{T}>0$ such that for any $\theta\in L^{\infty}(0,T; B_{p,\infty}^{1-\alpha + \frac{2}{p}})$, we have 
  \begin{equation*}
    {\left\|\int_{0}^{t}e^{-(t-s)\Lambda^{\alpha}}\Big(\big((\nabla ^\perp \Lambda^{-1}\theta)\cdot\nabla\big)\theta\Big) ~{\rm d}s\right\|}_{L^{\infty}(0,T; B_{\frac{p}{2}, 1}^{s})} 
    \leq C_{T}{\|\theta\|}_{L^\infty(0,T; B_{p,\infty}^{1-\alpha + \frac{2}{p}})}^2. 
  \end{equation*}
\end{lem}
\begin{proof}
  Since embedding $B_{p,\infty}^{s_{1}} \hookrightarrow B_{p,1}^{s_{2}}$ $(s_{1} > s_{2})$ holds in non-homogeneous Besov spaces, 
  we only consider the case 
  \begin{equation*}
    1 - 2\alpha + \frac{4}{p} 
    < s 
    < 1 - \alpha + \frac{4}{p}. 
  \end{equation*}
  By the boundedness of the fractional heat kernel, we get 
  \begin{equation*}
    {\left\|\psi*\left(e^{-(t-s)\Lambda^{\alpha}}\Big(\big((\nabla ^\perp \Lambda^{-1}\theta)\cdot\nabla\big)\theta\Big)\right)\right\|}_{L^{\frac{p}{2}}} 
    \leq C {\left\|\psi*\Big(\big((\nabla ^\perp \Lambda^{-1}\theta)\cdot\nabla\big)\theta\Big)\right\|}_{L^{\frac{p}{2}}}
  \end{equation*}
  For each $j\in \mathbb{N}$, using Lemma~\ref{0206-2},  we have 
  \begin{align*}
    &2^{sj}{\left\|\phi_j*\left(e^{-(t-s)\Lambda^{\alpha}}\Big(\big((\nabla ^\perp \Lambda^{-1}\theta)\cdot\nabla\big)\theta\Big)\right)\right\|}_{L^{\frac{p}{2}}}\\ 
    \leq& C2^{\left(s-1+2\alpha - \frac{4}{p}\right)j}e^{-c(t-s)2^{\alpha j}}2^{\left(1-2\alpha + \frac{4}{p}\right)j}{\left\|\phi_j*\Big(\big((\nabla ^\perp \Lambda^{-1}\theta)\cdot\nabla\big)\theta\Big)\right\|}_{L^{\frac{p}{2}}}. 
  \end{align*}
  Since $s>1-2\alpha+4/p$, we get 
  \begin{align*}
    \sum_{j\in \mathbb{N}}2^{\left(s-1+2\alpha - \frac{4}{p}\right)j}e^{-c(t-s)2^{\alpha j}} 
    =& (t-s)^{-\frac{1}{\alpha}\left(s-1+2\alpha - \frac{4}{p}\right)}\sum_{j\in \mathbb{N}}\big((t-s)2^{\alpha j}\big)^{\frac{1}{\alpha}\left(s-1+2\alpha - \frac{4}{p}\right)}e^{-c(t-s)2^{\alpha j}}\\ 
    \leq& C (t-s)^{-\frac{1}{\alpha}\left(s-1+2\alpha - \frac{4}{p}\right)}
  \end{align*}
  Thanks to $s<1-\alpha+4/p$, using Young's inequality, we obtain 
  \begin{equation*}
    {\left\|\int_{0}^{t}e^{-(t-s)\Lambda^{\alpha}}\Big(\big((\nabla ^\perp \Lambda^{-1}\theta)\cdot\nabla\big)\theta\Big) ~{\rm d}s\right\|}_{L^{\infty}(0,T; B_{\frac{p}{2}, 1}^{s})} 
    \leq C_{T}{\left\|\big((\nabla ^\perp \Lambda^{-1}\theta)\cdot\nabla\big)\theta\right\|}_{L^{\infty}(0,T; B_{\frac{p}{2}, \infty}^{1-2\alpha + \frac{4}{p}})}. 
  \end{equation*}
  Since $-\alpha<1-2\alpha + 4/p<0$, by Lemma~\ref{0205-5} \eqref{0212-9} and Lemma~\ref{0205-9}, we have 
  \begin{equation*}
    {\left\|\big((\nabla ^\perp \Lambda^{-1}\theta)\cdot\nabla\big)\theta\right\|}_{B_{\frac{p}{2}, \infty}^{1-2\alpha + \frac{4}{p}}} 
    \leq C {\|\theta\|}_{B_{p,\infty}^{1-\alpha + \frac{2}{p}}}^2. 
  \end{equation*}
\end{proof}

In the case $\alpha = 1$, the following proposition holds. 
%%$\alpha = 1$, $p$有限のときにエネルギークラスに入ること%%
\begin{prop}\label{0407-1}
  Let $\alpha = 1$, $T>0$ and $2\leq p < \infty$. 
  Let $\theta, \bar{\theta} \in L^{\infty}(0,T; B_{p, \infty}^{\frac{2}{p}})$ be solutions of the integral equation of \eqref{SQG} in the sense of Definition~\ref{0215-1} with the same initial data $\theta_0 \in B_{p, \infty}^{\frac{2}{p}}$. 
  Then we have 
  \begin{equation*}
    \theta - \bar{\theta} \in L^{\infty}(0,T; \dot{H}^{-\frac{1}{2}})\cap L^{2}(0,T; L^{2}). 
  \end{equation*}
\end{prop}
Proposition~\ref{0407-1} follows by the same argument as above, based on estimate of the Duhamel term and an iteration argument. 
Accordingly, as in Proposition~\ref{0211-4}, we only state the lemma used in the proof. 

%%$\alpha = 1$の場合のfirst iteration%% 
The next lemma follows from the same argument as Lemma~\ref{0410-1} and the fact that $B_{p,\infty}^{\frac{2}{p}} \hookrightarrow L^{p}$. 
\begin{lem}
  Let $T>0$, $1\leq p < \infty$ and $s\in\mathbb{R}$ with $s < 2/p$. 
  Then there exists a constant $C_{T}>0$ such that for any $\theta\in L^{\infty}(0,T; B_{p,\infty}^{\frac{2}{p}})$, we have 
  \begin{equation*}
    {\left\|\int_{0}^{t}e^{-(t-s)\Lambda}\Big(\big((\nabla ^\perp \Lambda^{-1}\theta)\cdot\nabla\big)\theta\Big) ~{\rm d}s\right\|}_{L^{\infty}(0,T; B_{\frac{p}{2}, 1}^{s})} 
    \leq C_{T}{\|\theta\|}_{L^\infty(0,T; B_{p,\infty}^{\frac{2}{p}})}^2. 
  \end{equation*}
\end{lem}

\section{Proof of theorems}\label{0327-2} 
\subsection{Proof of Theorem~\ref{0204-1}}
By the (global) existence result (\cite{Ca_Fe_2008}), 
there exists a solution $\bar{\theta} \in C([0,T] ; L^{\frac{2}{\alpha-1}})$ of \eqref{SQG} with $\bar{\theta}(0) = \theta_0$ and satisfies \eqref{0203-1}. 
We prove that 
\begin{equation*}
  \theta^{(1)}(t) = \bar{\theta}(t) \text{ and }
  \theta^{(2)}(t) = \bar{\theta}(t) \text{ in }L^{\frac{2}{\alpha-1}} 
  \text{ for a.e. }t\in (0, T). 
\end{equation*}
This implies Theorem~\ref{0204-1}. 
It is sufficient to show that $\theta^{(1)}(t) = \bar{\theta}(t)$. 
Let 
\begin{equation*}
  w := \theta^{(1)} - \bar{\theta}. 
\end{equation*}
Since $\theta^{(1)}, \bar{\theta}$ fulfill \eqref{0213-1} (see also Proposition~\ref{0204-4}), $w$ satisfies 
\begin{equation*}
  \begin{cases}
    \partial_t w + \Lambda^{\alpha} w + ((\nabla^{\perp}\Lambda^{-1}\theta^{(1)})\cdot\nabla)\theta^{(1)} - ((\nabla^{\perp}\Lambda^{-1}\bar{\theta})\cdot\nabla)\bar{\theta} = 0,\\ 
    w(0,x) = 0, 
  \end{cases}
\end{equation*}
in the sense of distributions, and we can write 
\begin{equation*}
  ((\nabla^{\perp}\Lambda^{-1}\theta^{(1)})\cdot\nabla)\theta^{(1)} - ((\nabla^{\perp}\Lambda^{-1}\bar{\theta})\cdot\nabla)\bar{\theta}
  = ((\nabla^{\perp}\Lambda^{-1}w)\cdot\nabla)\theta^{(1)} + ((\nabla^{\perp}\Lambda^{-1}\bar{\theta})\cdot\nabla)w
\end{equation*}
Moreover, using Proposition~\ref{0204-3}, $w \in L^{\infty}(0,T ; \dot{H}^{-\frac{1}{2}})\cap L^2(0,T ; \dot{H}^{\frac{\alpha}{2}-\frac{1}{2}})$. 
Thus, we can justify the energy inequality~\eqref{0215-6} i.e., $w$ satisfies 
\begin{align}\label{0212-4}
  &\frac{1}{2}\frac{{\rm d}}{{\rm d}t}{\|\Lambda^{-\frac{1}{2}}w(t)\|}_{L^2}^2 
  + {\|\Lambda^{\frac{\alpha}{2}-\frac{1}{2}}w(t)\|}_{L^2}^2\notag\\ 
  \leq& -\int_{\mathbb{R}^2}\Big(\big((\nabla^{\perp}\Lambda^{-1}w(t))\cdot \nabla\big)\theta^{(1)}(t) + \big((\nabla^{\perp}\Lambda^{-1}\bar{\theta}(t))\cdot \nabla\big) w(t)\Big)(\Lambda^{-1}w(t))~{\rm d}x. 
\end{align}
Note that thanks to $\nabla^{\perp} \cdot \nabla f =0$, we obtain 
\begin{align*}
  &\int_{\mathbb{R}^2}\big((\nabla^{\perp}\Lambda^{-1}w(t))\cdot \nabla\big)\theta^{(1)}(t)(\Lambda^{-1}w(t))~{\rm d}x\\ 
  =& - \int_{\mathbb{R}^2}\Big(\big((\nabla^{\perp}\Lambda^{-1}w(t))\cdot \nabla\big)\theta^{(1)}(t)\Big)\Lambda^{-1}w(t)~{\rm d}x, 
\end{align*}
and we have 
\begin{equation*}
  \int_{\mathbb{R}^2}\big((\nabla^{\perp}\Lambda^{-1}w(t))\cdot \nabla\big)\theta^{(1)}(t)(\Lambda^{-1}w(t))~{\rm d}x 
  = 0. 
\end{equation*}
We can also write 
\begin{align}\label{0204-5}
  &- \int_{\mathbb{R}^2}\big((\nabla^{\perp}\Lambda^{-1}\bar{\theta}(t))\cdot \nabla\big)w(t)(\Lambda^{-1}w(t))~{\rm d}x\notag\\ 
  =& \int_{\mathbb{R}^2}(\Lambda^{-1}\bar{\theta}(t))\Big(\big((\nabla^{\perp}\Lambda^{-1}w(t))\cdot \nabla\big)w(t)\Big)~{\rm d}x. 
\end{align}
We now estimate the right-hand side of \eqref{0204-5}. 
Let $N\in\mathbb{N}$ be fixed later. 
We divide $\bar{\theta}$ as follows, 
\begin{equation*}
  \bar{\theta} = \sum_{|j|>N} \phi_j*\bar{\theta} + \sum_{|j|\leq N} \phi_j*\bar{\theta} 
  =: {\bar{\theta}}_{>N} + {\bar{\theta}}_{\leq N}. 
\end{equation*}
By Lemma~\ref{0205-1}, Lemma~\ref{0205-2} and continuity of $\bar{\theta}$ with respect to time, 
for any $\varepsilon>0$, there exists $N\in \mathbb{N}$ such that 
\begin{equation*}
  \sup_{t\in (0, T)}{\|{\bar{\theta}}_{>N}(t)\|}_{L^{\frac{2}{\alpha-1}}}\leq \varepsilon, 
\end{equation*}
and there exists a constant $C>0$ such that 
\begin{equation*}
  \sup_{t\in (0, T)}{\|{\bar{\theta}}_{\leq N}(t)\|}_{W^{\alpha - \frac{1}{2}, \frac{2}{\alpha-1}}}\leq C. 
\end{equation*}
By H\"older's inequality, we have 
\begin{align*}
  &\int_{\mathbb{R}^2}(\Lambda^{-1}\bar{\theta}(t))\Big(\big((\nabla^{\perp}\Lambda^{-1}w(t))\cdot \nabla\big)w(t)\Big)~{\rm d}x\\ 
  \leq& \sup_{t\in (0, T)}{\|{\bar{\theta}}_{>N}(t)\|}_{L^{\frac{2}{\alpha-1}}}{\|\big((\nabla^{\perp}\Lambda^{-1}w(t))\cdot \nabla\big)w(t)\|}_{\dot{W}^{-1, \frac{2}{3-\alpha}}}\\
  &+\sup_{t\in (0, T)}{\|{\bar{\theta}}_{\leq N}(t)\|}_{W^{\alpha - \frac{1}{2}, \frac{2}{\alpha-1}}}{\|\big((\nabla^{\perp}\Lambda^{-1}w(t))\cdot \nabla\big)w(t)\|}_{W^{-\alpha - \frac{1}{2}, \frac{2}{3-\alpha}}}\\ 
  \leq& \varepsilon {\|\big((\nabla^{\perp}\Lambda^{-1}w(t))\cdot \nabla\big)w(t)\|}_{\dot{W}^{-1, \frac{2}{3-\alpha}}} 
   + C{\|\big((\nabla^{\perp}\Lambda^{-1}w(t))\cdot \nabla\big)w(t)\|}_{W^{-\alpha - \frac{1}{2}, \frac{2}{3-\alpha}}}. 
\end{align*}
Using $\nabla\cdot \nabla^{\perp}=0$ and Sobolev embedding $\dot{W}^{\frac{\alpha}{2}-\frac{1}{2}, p}(\mathbb{R}^2)\hookrightarrow L^{\frac{2}{3-\alpha}}(\mathbb{R}^2)$, where 
\begin{equation*}
  \frac{1}{p} = \frac{3-\alpha}{2} + \frac{1}{2}\left(\frac{\alpha}{2} - \frac{1}{2}\right) = \frac{5-\alpha}{4}, 
\end{equation*} 
we get 
\begin{equation*}
  {\|\big((\nabla^{\perp}\Lambda^{-1}w)\cdot \nabla\big)w\|}_{\dot{W}^{-1, \frac{2}{3-\alpha}}} 
  \leq C{\|(\nabla^{\perp}\Lambda^{-1}w)w\|}_{\dot{W}^{\frac{\alpha}{2}-\frac{1}{2}, \frac{4}{5-\alpha}}}. 
\end{equation*}
Now $0<\alpha/2 - 1/2$ and $1<4/(5-\alpha)$. 
Thus, by Lemma~\ref{0205-3} and boundedness of Riesz transform in $L^p$ ($1<p<\infty$), we obtain 
\begin{align*}
  {\|(\nabla^{\perp}\Lambda^{-1}w)w\|}_{\dot{W}^{\frac{\alpha}{2}-\frac{1}{2}, \frac{4}{5-\alpha}}} 
  &\leq C({\|\nabla^{\perp}\Lambda^{-1}w\|}_{L^{\frac{4}{3-\alpha}}}{\|w\|}_{\dot{H}^{\frac{\alpha}{2}-\frac{1}{2}}}
   + {\|\nabla^{\perp}\Lambda^{-1}w\|}_{\dot{H}^{\frac{\alpha}{2}-\frac{1}{2}}}{\|w\|}_{L^{\frac{4}{3-\alpha}}})\\ 
  &\leq C {\|w\|}_{L^{\frac{4}{3-\alpha}}}{\|w\|}_{\dot{H}^{\frac{\alpha}{2}-\frac{1}{2}}}. 
\end{align*}
Since $2<4/(3-\alpha)$, we have 
\begin{equation*}
  \dot{W}^{\frac{\alpha}{2} - \frac{1}{2}, p}(\mathbb{R}^2)\hookrightarrow L^{\frac{4}{3-\alpha}}(\mathbb{R}^2), 
  \text{ where } \frac{1}{p} = \frac{3-\alpha}{4} + \frac{1}{2}\left(\frac{\alpha}{2} - \frac{1}{2}\right) 
  = \frac{1}{2}. 
\end{equation*}
Hence we obtain 
\begin{equation*}
  {\|\big((\nabla^{\perp}\Lambda^{-1}w(t))\cdot \nabla\big)w(t)\|}_{\dot{W}^{-1, \frac{2}{3-\alpha}}} 
  \leq C {\|w(t)\|}_{\dot{H}^{\frac{\alpha}{2}-\frac{1}{2}}}^2. 
\end{equation*}
We now estimate 
\begin{equation*}
  {\|\big((\nabla^{\perp}\Lambda^{-1}w(t))\cdot \nabla\big)w(t)\|}_{W^{-\alpha - \frac{1}{2}, \frac{2}{3-\alpha}}}. 
\end{equation*}
Using Lemma~\ref{0205-4} \eqref{0206-8}, we have 
\begin{equation*}
  {\|\big((\nabla^{\perp}\Lambda^{-1}w)\cdot \nabla\big)w\|}_{W^{-\alpha - \frac{1}{2} , \frac{2}{3-\alpha}}} 
  \leq C {\|\big((\nabla^{\perp}\Lambda^{-1}w)\cdot \nabla\big)w\|}_{B_{\frac{2}{3-\alpha}, 1}^{-\alpha - \frac{1}{2}}}. 
\end{equation*}
By Bony's decomposition, we get 
\begin{align*}
  &{\|\big((\nabla^{\perp}\Lambda^{-1}w)\cdot \nabla\big)w\|}_{B_{\frac{2}{3-\alpha}, 1}^{-\alpha - \frac{1}{2}}}\\ 
  \leq& {\left\|\sum_{l\geq 2}\big((\nabla^{\perp}\Lambda^{-1}(\psi*w))\cdot \nabla\big)(\phi_l*w) + \sum_{k\leq l - 2}\big((\nabla^{\perp}\Lambda^{-1}(\phi_k*w))\cdot \nabla\big)(\phi_l*w)\right\|}_{B_{\frac{2}{3-\alpha}, 1}^{-\alpha - \frac{1}{2}}} \\ 
    &+ {\left\|\sum_{k\geq 2}\big((\nabla^{\perp}\Lambda^{-1}(\phi_{k}*w))\cdot \nabla\big)(\psi*w) + \sum_{l\leq k - 2}\big((\nabla^{\perp}\Lambda^{-1}(\phi_k*w))\cdot \nabla\big)(\phi_l*w)\right\|}_{B_{\frac{2}{3-\alpha}, 1}^{-\alpha - \frac{1}{2}}}\\ 
    &+ {\left\|\big((\nabla^{\perp}\Lambda^{-1}(\psi*w))\cdot \nabla\big)(\psi*w)\right\|}_{B_{\frac{2}{3-\alpha}, 1}^{-\alpha - \frac{1}{2}}}\\ 
    &+ {\left\|\big((\nabla^{\perp}\Lambda^{-1}(\psi*w))\cdot \nabla\big)(\phi_{1}*w) + \big((\nabla^{\perp}\Lambda^{-1}(\phi_{1}*w))\cdot \nabla\big)(\psi*w)\right\|}_{B_{\frac{2}{3-\alpha}, 1}^{-\alpha - \frac{1}{2}}}\\ 
    &+ {\left\|\sum_{|k-l|\leq 1}\big((\nabla^{\perp}\Lambda^{-1}(\phi_k*w))\cdot \nabla\big)(\phi_l*w)\right\|}_{B_{\frac{2}{3-\alpha}, 1}^{-\alpha - \frac{1}{2}}}. 
\end{align*}
Using Lemma~\ref{0205-7}, product estimate (Lemma~\ref{0205-5} \eqref{0212-9}) and embedding (Lemma~\ref{0205-6}), we have 
\begin{align*}
  &{\left\|\sum_{l\geq 2}\big((\nabla^{\perp}\Lambda^{-1}(\psi*w))\cdot \nabla\big)(\phi_l*w) + \sum_{k\leq l - 2}\big((\nabla^{\perp}\Lambda^{-1}(\phi_k*w))\cdot \nabla\big)(\phi_l*w)\right\|}_{B_{\frac{2}{3-\alpha}, 1}^{-\alpha - \frac{1}{2}}}\\ 
  \leq& C {\left\|\sum_{l\geq 2}(\nabla^{\perp}\Lambda^{-1}(\psi*w))(\phi_l*w) + \sum_{k\leq l - 2}(\nabla^{\perp}\Lambda^{-1}(\phi_k*w))(\phi_l*w)\right\|}_{B_{\frac{2}{3-\alpha}, 1}^{-\alpha + \frac{1}{2}}}\\ 
  \leq& C {\|\nabla^{\perp}\Lambda^{-1}w\|}_{B_{\frac{2}{2-\alpha}, 2}^{-\alpha + \frac{1}{2}}}{\|w\|}_{B_{2,2}^{0}}\\ 
  \leq& C {\|w\|}_{B_{2, 2}^{-\frac{1}{2}}}{\|w\|}_{L^{2}}, 
\end{align*}
and 
\begin{align*}
  &{\left\|\sum_{k\geq 2}\big((\nabla^{\perp}\Lambda^{-1}(\phi_{k}*w))\cdot \nabla\big)(\psi*w) + \sum_{l\leq k - 2}\big((\nabla^{\perp}\Lambda^{-1}(\phi_k*w))\cdot \nabla\big)(\phi_l*w)\right\|}_{B_{\frac{2}{3-\alpha}, 1}^{-\alpha - \frac{1}{2}}}\\ 
  \leq& C {\|w\|}_{B_{2, 2}^{-\frac{1}{2}}}{\|w\|}_{L^{2}}. 
\end{align*}
Since it can be written as a finite sum, we obtain 
\begin{equation*}
  {\left\|\big((\nabla^{\perp}\Lambda^{-1}(\psi*w))\cdot \nabla\big)(\psi*w)\right\|}_{B_{\frac{2}{3-\alpha}, 1}^{-\alpha - \frac{1}{2}}} 
  \leq C {\|w\|}_{B_{2, 2}^{-\frac{1}{2}}}{\|w\|}_{L^{2}}, 
\end{equation*}
and 
\begin{align*}
  &{\left\|\big((\nabla^{\perp}\Lambda^{-1}(\psi*w))\cdot \nabla\big)(\phi_{1}*w) + \big((\nabla^{\perp}\Lambda^{-1}(\phi_{1}*w))\cdot \nabla\big)(\psi*w)\right\|}_{B_{\frac{2}{3-\alpha}, 1}^{-\alpha - \frac{1}{2}}}\\ 
  \leq& C {\|w\|}_{B_{2, 2}^{-\frac{1}{2}}}{\|w\|}_{L^{2}}. 
\end{align*}
Note that $B_{2,2}^{s} = H^s$ (Lemma~\ref{0205-8}). 
Using Lemma~\ref{0205-6}, we get 
\begin{align*}
  &{\left\|\sum_{|k-l|\leq 2}\big((\nabla^{\perp}\Lambda^{-1}(\phi_k*w))\cdot \nabla\big)(\phi_l*w)\right\|}_{B_{\frac{2}{3-\alpha}, 1}^{-\alpha - \frac{1}{2}}}\\ 
  \leq& C {\left\|\sum_{|k-l|\leq 2}\big((\nabla^{\perp}\Lambda^{-1}(\phi_k*w))\cdot \nabla\big)(\phi_l*w)\right\|}_{B_{1, 1}^{-\frac{3}{2}}}.  
\end{align*}
Since we can write 
\begin{align*}
  &\sum_{|k-l|\leq 1}\big((\nabla^{\perp}\Lambda^{-1}(\phi_k*w))\cdot \nabla\big)(\phi_l*w)\\ 
  =& \frac{1}{2}\sum_{|k-l|\leq 1}\Big(\big((\nabla^{\perp}\Lambda^{-1}(\phi_k*w))\cdot \nabla\big)(\phi_l*w) 
  + \big((\nabla^{\perp}\Lambda^{-1}(\phi_l*w))\cdot \nabla\big)(\phi_k*w)\Big), 
\end{align*}
by Lemma~\ref{0205-9}, we get 
\begin{equation*}
  {\left\|\sum_{|k-l|\leq 2}\big((\nabla^{\perp}\Lambda^{-1}(\phi_k*w))\cdot \nabla\big)(\phi_l*w)\right\|}_{B_{1, 1}^{-\frac{3}{2}}} 
  \leq C {\|w\|}_{L^{2}}{\|w\|}_{H^{- \frac{1}{2}}}. 
\end{equation*}
Hence, we have 
\begin{equation*}
  \frac{1}{2}\frac{{\rm d}}{{\rm d}t}{\|\Lambda^{-\frac{1}{2}}w(t)\|}_{L^2}^2 
  + {\|\Lambda^{\frac{\alpha}{2}-\frac{1}{2}}w(t)\|}_{L^2}^2 
  \leq C\varepsilon{\|w(t)\|}_{\dot{H}^{\frac{\alpha}{2}-\frac{1}{2}}}^2 + C {\|w(t)\|}_{L^{2}}{\|w(t)\|}_{H^{- \frac{1}{2}}}. 
\end{equation*}
Thanks to $\dot{H}^{-s} \hookrightarrow H^{-s}$ $(s>0)$ and interpolation, we obtain 
\begin{equation*}
  {\|w(t)\|}_{L^{2}}{\|w(t)\|}_{H^{- \frac{1}{2}}} 
  \leq C {\|w(t)\|}_{\dot{H}^{- \frac{1}{2}}}^{\theta}{\|w(t)\|}_{\dot{H}^{\frac{\alpha}{2} - \frac{1}{2}}}^{1 - \theta}{\|w(t)\|}_{\dot{H}^{- \frac{1}{2}}}, 
\end{equation*}
where $\theta\in \mathbb{R}$ satisfies 
\begin{equation*}
  -\frac{\theta}{2} + (1-\theta)\left(\frac{\alpha}{2} - \frac{1}{2}\right) = 0. 
\end{equation*}
Note that $0<\theta<1$. 
Moreover, we have 
\begin{equation*}
  {\|w(t)\|}_{\dot{H}^{- \frac{1}{2}}}^{\theta}{\|w(t)\|}_{\dot{H}^{\frac{\alpha}{2} - \frac{1}{2}}}^{1 - \theta}{\|w(t)\|}_{\dot{H}^{- \frac{1}{2}}} 
  \leq \frac{C}{\varepsilon}{\|w(t)\|}_{\dot{H}^{- \frac{1}{2}}}^{2} + \varepsilon{\|w(t)\|}_{\dot{H}^{\frac{\alpha}{2} - \frac{1}{2}}}^{2}. 
\end{equation*}
Thus, we obtain 
\begin{align*}
  \frac{1}{2}\frac{{\rm d}}{{\rm d}t}{\|\Lambda^{-\frac{1}{2}}w(t)\|}_{L^2}^2 
  + {\|\Lambda^{\frac{\alpha}{2}-\frac{1}{2}}w(t)\|}_{L^2}^2 
  \leq C \varepsilon {\|w(t)\|}_{\dot{H}^{\frac{\alpha}{2}-\frac{1}{2}}}^2 + \frac{C}{\varepsilon}{\|w(t)\|}_{\dot{H}^{- \frac{1}{2}}}^{2}. 
\end{align*}
We take $\varepsilon>0$ sufficiently small such that 
\begin{equation*}
  C\varepsilon  \leq \frac{1}{2}, 
\end{equation*}
then we obtain 
\begin{equation*}
  \frac{{\rm d}}{{\rm d}t}{\|\Lambda^{-\frac{1}{2}}w(t)\|}_{L^2}^2 
  \leq \frac{C}{\varepsilon}{\|\Lambda^{-\frac{1}{2}}w(t)\|}_{L^2}^2. 
\end{equation*}
Using Gronwall's inequality, we have 
\begin{equation*}
  w(t)=0 \text{ in } \dot{H}^{-\frac{1}{2}} \text{ for a.e. } t\in (0,T), 
\end{equation*}
also, 
\begin{equation*}
  w(t)=0 \text{ in } L^{\frac{2}{\alpha-1}} \text{ for a.e. } t\in (0,T). 
\end{equation*}

\subsection{Proof of Theorem~\ref{0212-1}} 
Similarly to the proof of Theorem~\ref{0204-1}, 
let $\bar{\theta}\in C([0,T]; B_{\frac{4}{\alpha-1}, 2}^{-\frac{1}{2}(\alpha - 1)})$ be a solution of the integral equation of \eqref{SQG} in the sense of Definition~\ref{0215-1} with $\bar{\theta}(0)=\theta_0$. 
Then by Proposition~\ref{0211-1}, $w = \theta^{(1)} - \bar{\theta}$ belongs to $L^{\infty}(0,T; \dot{H}^{-\frac{1}{2}})\cap L^{2}(0,T; \dot{H}^{\frac{\alpha}{2} - \frac{1}{2}})$ and satisfies \eqref{0212-4}. 
Moreover, by duality in Besov spaces (Lemma~\ref{0212-5}), Lemma~\ref{0212-2} and Lemma~\ref{0212-3}, we have 
\begin{align*}
  &\int_{\mathbb{R}^2}(\Lambda^{-1}\bar{\theta}(t))\Big(\big((\nabla^{\perp}\Lambda^{-1}w(t))\cdot \nabla\big)w(t)\Big)~{\rm d}x\\ 
  \leq& \sup_{t\in (0, T)}{\|{\bar{\theta}}_{>N}(t)\|}_{B_{\frac{4}{\alpha-1}, 2}^{-\frac{1}{2}(\alpha - 1)}}{\|\big((\nabla^{\perp}\Lambda^{-1}w(t))\cdot \nabla\big)w(t)\|}_{B_{\frac{4}{5-\alpha}, 2}^{\frac{1}{2}(\alpha - 1) - 1}}\\
  &+\sup_{t\in (0, T)}{\|{\bar{\theta}}_{\leq N}(t)\|}_{B_{\frac{4}{\alpha-1}, 2}^{\frac{\alpha}{2}}}{\|\big((\nabla^{\perp}\Lambda^{-1}w(t))\cdot \nabla\big)w(t)\|}_{B_{\frac{4}{5-\alpha}, 2}^{-\frac{\alpha}{2} - 1}}\\ 
  \leq& \varepsilon {\left\|\big((\nabla^{\perp}\Lambda^{-1}w(t))\cdot \nabla\big)w(t)\right\|}_{B_{\frac{4}{5-\alpha}, 2}^{\frac{1}{2}(\alpha - 1) - 1}} 
   + C{\left\|\big((\nabla^{\perp}\Lambda^{-1}w(t))\cdot \nabla\big)w(t)\right\|}_{B_{\frac{4}{5-\alpha}, 2}^{-\frac{\alpha}{2} - 1}}. 
\end{align*}
Since $\nabla\cdot\nabla^{\perp}f=0$ and $\alpha>1$, using Lemma~\ref{0205-5} \eqref{0212-6} and Lemma~\ref{0212-7}, we get 
\begin{equation*}
  {\left\|\big((\nabla^{\perp}\Lambda^{-1}w)\cdot \nabla\big)w\right\|}_{B_{\frac{4}{5-\alpha}, 2}^{\frac{1}{2}(\alpha - 1) - 1}} 
  \leq C{\|w\|}_{L^{\frac{4}{3-\alpha}}}{\|w\|}_{B_{2,2}^{\frac{\alpha}{2}-\frac{1}{2}}}. 
\end{equation*}
Thanks to $B_{2,2}^{s} = H^{s}$, $H^{s} = L^{2}\cap \dot{H}^{s}$ $(s>0)$ and interpolation, we obtain 
\begin{equation*}
  {\|w\|}_{B_{2,2}^{\frac{\alpha}{2}-\frac{1}{2}}} 
  \leq C\left({\|w\|}_{L^{2}} + {\|w\|}_{\dot{H}^{\frac{\alpha}{2}-\frac{1}{2}}}\right) 
  \leq C\left({\|w\|}_{\dot{H}^{-\frac{1}{2}}} + {\|w\|}_{\dot{H}^{\frac{\alpha}{2}-\frac{1}{2}}}\right). 
\end{equation*}
Note that $3\leq 4/(3-\alpha)$, by Sobolev embedding, we have 
\begin{equation*}
  {\|w\|}_{L^{\frac{4}{3-\alpha}}} 
  \leq C {\|w\|}_{\dot{H}^{\frac{\alpha}{2}- \frac{1}{2}}}. 
\end{equation*}
Thus, we get 
\begin{equation*}
  {\left\|\big((\nabla^{\perp}\Lambda^{-1}w)\cdot \nabla\big)w\right\|}_{B_{\frac{4}{5-\alpha}, 2}^{\frac{1}{2}(\alpha - 1) - 1}} 
  \leq C\left({\|w\|}_{\dot{H}^{-\frac{1}{2}}}^{2} + {\|w\|}_{\dot{H}^{\frac{\alpha}{2}-\frac{1}{2}}}^{2}\right). 
\end{equation*}
Also, using Lemma~\ref{0205-9} instead of Lemma~\ref{0212-7}, we obtain 
\begin{equation*}
  {\left\|\big((\nabla^{\perp}\Lambda^{-1}w(t))\cdot \nabla\big)w(t)\right\|}_{B_{\frac{4}{5-\alpha}, 2}^{-\frac{\alpha}{2} - 1}} 
  \leq C {\|w\|}_{L^{2}}{\|w\|}_{B_{2,2}^{-\frac{1}{2}}}. 
\end{equation*}
Therefore, by the same argument as in Theorem~\ref{0204-1}, we can prove that the solution of \eqref{SQG} is unique. 

\subsection{Proof of Theorem~\ref{0212-8}} 
In this case, we consider 
\begin{equation*}
  \theta_{0} \in B_{p,p}^{-\frac{2}{3} + \frac{2}{p}} \text{ if } \alpha = \frac{5}{3} \text{ and } 3\leq p < 6, 
\end{equation*}
or 
\begin{equation*}
  \theta_{0} \in B_{p,q}^{1 - \alpha + \frac{2}{p}} \text{ if } 1 < \alpha < \frac{5}{2}, \frac{2}{\alpha - 1} \leq p < \frac{4}{\alpha - 1} \text{ and } 1\leq q < \infty. 
\end{equation*}
The proof is carried out using the same method as before. 
The energy inequality is justified by Proposition~\ref{0211-4}. 
We need to estimate the following term 
\begin{equation}\label{0409-2}
  {\left\|\big((\nabla^{\perp}\Lambda^{-1}w)\cdot \nabla\big)w\right\|}_{B_{\frac{p}{p-1}, \frac{q}{q-1}}^{-2 + \alpha - \frac{2}{p}}} 
  \quad\text{ and }\quad 
  {\left\|\big((\nabla^{\perp}\Lambda^{-1}w)\cdot \nabla\big)w\right\|}_{B_{\frac{p}{p-1}, \frac{q}{q-1}}^{- \frac{2}{p} - \frac{3}{2}}}. 
\end{equation}
Thanks to $-1+\alpha-2/p<\alpha/2-1/2$, 
by product estimate (Lemma~\ref{0205-5} \eqref{0212-9} and Lemma~\ref{0212-7} or Lemma~\ref{0205-9}) and embedding (Lemma~\ref{0205-6}), we obtain 
\begin{equation*}
  {\left\|\big((\nabla^{\perp}\Lambda^{-1}w)\cdot \nabla\big)w\right\|}_{B_{\frac{p}{p-1}, \frac{q}{q-1}}^{-2 + \alpha - \frac{2}{p}}} 
  \leq C {\|w\|}_{B_{\frac{2p}{p-2}, \frac{2q}{q-2}}^{\frac{\alpha}{2} - \frac{1}{2} - \frac{2}{p}}}{\|w\|}_{B_{2, 2}^{\frac{\alpha}{2} - \frac{1}{2}}} 
  \leq C {\|w\|}_{H^{\frac{\alpha}{2} - \frac{1}{2}}}^2, 
\end{equation*}
and 
\begin{equation*}
  {\left\|\big((\nabla^{\perp}\Lambda^{-1}w)\cdot \nabla\big)w\right\|}_{B_{\frac{p}{p-1}, \frac{q}{q-1}}^{- \frac{2}{p} - \frac{3}{2}}} 
  \leq C {\|w\|}_{B_{\frac{2p}{p-2}, \frac{2q}{q-2}}^{- \frac{2}{p}}}{\|w\|}_{B_{2, 2}^{-\frac{1}{2}}} 
  \leq C {\|w\|}_{L^{2}}{\|w\|}_{H^{- \frac{1}{2}}}. 
\end{equation*}
Thus, we can show that $w(t) = 0 $ in $B_{p,q}^{1-\alpha + \frac{2}{p}}$ for a.e. $t\in (0,T)$. 

\subsection{Proof of Theorem~\ref{0409-1}} 
In Theorem~\ref{0409-1}, we consider the cases $\alpha = 1$ and $\theta_{0}\in B_{p,q}^{\frac{2}{p}}$ with $2\leq p < \infty$, $1\leq q < \infty$. 
In this case, local well-posedness is established by Chen, Miao and Zhang~\cite{Ch_Mi_Zh_2007}. 
Furthermore, by the argument in Wang and Zhang~\cite{Wa_Za_2011} using modulus of continuity, the solution exists globally in time (see also~\cite{Ki_Na_Vo_2007}). 
Hence, the same method as before can be applied. 
By Proposition~\ref{0407-1}, the difference between two solutions with the same initial data $\theta_{0}\in B_{p,q}^{\frac{2}{p}}$ belongs to $L^{\infty}(0,T; \dot{H}^{-\frac{1}{2}})\cap L^{2}(0,T; L^{2})$. 
Using the energy method, the proof reduces to the estimate of the following term. 
\begin{equation*}
  {\left\|\big((\nabla^{\perp}\Lambda^{-1}w)\cdot \nabla\big)w\right\|}_{B_{\frac{p}{p-1}, \frac{q}{q-1}}^{-\frac{2}{p} - 1}} 
  \quad\text{ and }\quad 
  {\left\|\big((\nabla^{\perp}\Lambda^{-1}w)\cdot \nabla\big)w\right\|}_{B_{\frac{p}{p-1}, \frac{q}{q-1}}^{- \frac{2}{p} - \frac{3}{2}}}. 
\end{equation*}
These terms can be estimated by the same argument as that used for terms~\eqref{0409-2}. 
Thus, we get 
\begin{equation*}
  {\left\|\big((\nabla^{\perp}\Lambda^{-1}w)\cdot \nabla\big)w\right\|}_{B_{\frac{p}{p-1}, \frac{q}{q-1}}^{-\frac{2}{p} - 1}} 
  \leq C {\|w\|}_{L^{2}}^{2}, 
\end{equation*}
and 
\begin{equation*}
  {\left\|\big((\nabla^{\perp}\Lambda^{-1}w)\cdot \nabla\big)w\right\|}_{B_{\frac{p}{p-1}, \frac{q}{q-1}}^{- \frac{2}{p} - \frac{3}{2}}} 
  \leq C {\|w\|}_{L^{2}}{\|w\|}_{H^{- \frac{1}{2}}}. 
\end{equation*}

Therefore, the uniqueness of the solutions of the integral equation of \eqref{SQG} in the sense of Definition~\ref{0215-1} in $L^{\infty}(0,T; B_{p,q}^{\frac{2}{p}})$ holds.

\appendix 
\section{Proof of Proposition~\ref{0204-4}. }\label{0213-2}
In this appendix, we prove Proposition~\ref{0204-4} using Littlewood-Paley decomposition. 
The overall strategy of the argument follows the proof of the analogous result for Navier-Stokes equations due to~\cite{Fa_Jo_Ri_1972}. 

\begin{proof}
  First, we assume that $\theta \in L^\infty(0,T; L^{\frac{2}{\alpha-1}})$ is a solution of \eqref{SQG} in the sense of Definition~\ref{0215-1}. 
  Fix $t\in (0,T)$ and $x\in\mathbb{R}^2$. 
  For any $j\in\mathbb{Z}$, by Definition~\ref{0215-1} \eqref{0203-1}, for $\phi_j(x - \cdot) \in \mathcal{S}(\mathbb{R}^2)$, 
  we obtain 
  \begin{equation*}
    \phi_j*\theta(t) = \phi_j*e^{-t\Lambda^{\alpha}}\theta_0 - \nabla\phi_j*\left(\int_{0}^{t}e^{-(t-s)\Lambda^{\alpha}}({\bf u}\theta)~{\rm d}s\right). 
  \end{equation*}
  Let $N\in \mathbb{N}$. 
  We denote 
  \begin{equation*}
    f_{N} := \sum_{|j|\leq N}\phi_j* f. 
  \end{equation*}
  Then $\theta_N$ satisfies 
  \begin{equation*}
    \theta_N(t) = e^{-t\Lambda^{\alpha}}\theta_{0,N} - \sum_{|j|\leq N}\nabla\phi_j*\left(\int_{0}^{t}e^{-(t-s)\Lambda^{\alpha}}({\bf u}\theta)~{\rm d}s\right). 
  \end{equation*}
  Note that for any $1\leq p <\infty$, we have 
  \begin{equation*}
    \lim_{N \to \infty}\theta_N = \theta \text{ in }L^p(0,T;L^{\frac{2}{\alpha-1}}) 
    \text{ and }
    \lim_{N \to \infty}({\bf u}\theta)_N = {\bf u}\theta \text{ in }L^{\frac{p}{2}}(0,T;L^{\frac{1}{\alpha-1}}). 
  \end{equation*}
  Since $\alpha>1$ and $\theta\in L^{\frac{2}{\alpha-1}}((0,T)\times \mathbb{R}^2)$, 
  there exists some sequence $\{\theta^{(\varepsilon)}\}_{\varepsilon>0}\in C_{c}^{\infty}((0,T)\times \mathbb{R}^2)$ such that 
  \begin{equation}\label{0417-1}
    \lim_{\varepsilon\to 0}\theta^{(\varepsilon)} = \theta \text{ in }L^{\frac{2}{\alpha-1}}((0,T) \times \mathbb{R}^2). 
  \end{equation}
  We consider $\theta_{N}^{(\varepsilon)}$ defined by 
  \begin{equation*}
    \theta_{N}^{(\varepsilon)} (t) 
    := e^{-t\Lambda^{\alpha}}\theta_{0,N} 
    - \sum_{|j|\leq N}\nabla\phi_j*\left(\int_{0}^{t}e^{-(t-s)\Lambda^{\alpha}}({\bf u}^{(\varepsilon)}\theta^{(\varepsilon)})~{\rm d}s\right). 
  \end{equation*}
  Since ${\bf u}^{(\varepsilon)}$ is also smooth, we can write 
  \begin{equation*}
    \sum_{|j|\leq N}\nabla\phi_j*\left(\int_{0}^{t}e^{-(t-s)\Lambda^{\alpha}}({\bf u}^{(\varepsilon)}\theta^{(\varepsilon)})~{\rm d}s\right) 
    = \int_{0}^{t}e^{-(t-s)\Lambda^{\alpha}}\big(\nabla\cdot{({\bf u}^{(\varepsilon)}\theta^{(\varepsilon)})}_N\big)~{\rm d}s. 
  \end{equation*}
  Then $\theta_{N}^{(\varepsilon)}$ is differentiable with respect to time in the strong sense, 
  and $\partial_t \theta_{N}^{(\varepsilon)}$ satisfies 
  \begin{equation*}
    \begin{cases}
      \partial_t \theta_{N}^{(\varepsilon)} 
      + \Lambda^{\alpha}\theta_{N}^{(\varepsilon)} + \nabla\cdot{({\bf u}^{(\varepsilon)}\theta^{(\varepsilon)})}_N =0 ,\\ 
      \theta_{N}^{(\varepsilon)}(0, x) = \theta_{0,N}(x). 
    \end{cases}
  \end{equation*}
  Thus, for any $\varphi \in C^\infty([0,T]\times \mathbb{R}^2)$ such that $\varphi(T)=0$, 
  and for any $t\in [0,T]$, $\varphi(t, \cdot) \in \mathcal{S}(\mathbb{R}^2)$, we obtain 
  \begin{equation*}
    \displaystyle\int_{0}^{T}\hspace{-6pt}\Big(\hspace{-2pt}\subscripts{\mathcal{S}'}{\displaystyle\big\langle\theta_N^{(\varepsilon)}(t), \partial_t \varphi(t) - \Lambda^{\alpha}\varphi(t) \big\rangle}{\mathcal{S}} 
      + \subscripts{\mathcal{S}'}{\displaystyle\big\langle {({\bf u}^{(\varepsilon)}(t)\theta^{(\varepsilon)}(t))}_N, \nabla\varphi(t) \big\rangle}{\mathcal{S}}\hspace{-2pt}\Big){\rm d}t
      + \hspace{-2pt} \subscripts{\mathcal{S}'}{\displaystyle\big\langle\theta_{0,N}, \varphi(0) \big\rangle}{\mathcal{S}} \hspace{-3pt} =  0. 
  \end{equation*}
  If necessary, taking subsequences, and passing to the limit with respect to $\varepsilon\to 0$ and $N\to \infty$, 
  we obtain 
  \begin{equation*}
    \displaystyle\int_{0}^{T}\hspace{-6pt}\Big(\hspace{-2pt}\subscripts{\mathcal{S}'}{\displaystyle\big\langle\theta(t), \partial_t \varphi(t) - \Lambda^{\alpha}\varphi(t) \big\rangle}{\mathcal{S}} 
      + \subscripts{\mathcal{S}'}{\displaystyle\big\langle {\bf u}(t)\theta(t), \nabla\varphi(t) \big\rangle}{\mathcal{S}}\hspace{-2pt}\Big){\rm d}t
      + \hspace{-2pt} \subscripts{\mathcal{S}'}{\displaystyle\big\langle\theta_{0}, \varphi(0) \big\rangle}{\mathcal{S}} \hspace{-3pt} =  0, 
  \end{equation*}
  which implies $\theta$ is a weak solution of \eqref{SQG}. 

  Conversely, we assume that $\theta \in L^\infty(0,T; L^{\frac{2}{\alpha-1}})$ is a weak solution of \eqref{SQG}. 
  Similarly, for fixed $t\in (0,T)$ and $x\in\mathbb{R}^2$, we choose $\varphi(t, y) = \psi(t)\phi_j(x-y)$, where $\psi\in C^{\infty}([0,T])$ with $\psi(T)=0$, 
  and approximate $\theta$ by smooth function $\theta^{(\varepsilon)}$ as in \eqref{0417-1}. 
  Then 
  \begin{equation*}
    \left(\theta^{(\varepsilon)}\right)_{\leq N} 
    := \sum_{|j| \leq N}\phi_{j}*\theta^{(\varepsilon)}
  \end{equation*}
  satisfies 
  \begin{equation*}
    \int_{0}^{T}\left(\theta^{(\varepsilon)}\right)_{\leq N}\partial_t\psi(t) 
    - \left(\Lambda^{\alpha}\left(\theta^{(\varepsilon)}\right)_{\leq N} + \nabla\cdot\left({\bf u}^{(\varepsilon)}(t)\theta^{(\varepsilon)}(t)\right)_{N}\right)\psi(t) ~{\rm d}t 
    + \theta_{0, N}=0. 
  \end{equation*} 
  Now, we can define $(\theta^{(\varepsilon)})_{\leq N}(0)$ such that $(\theta^{(\varepsilon)})_{\leq N}(0) \to \theta_{0, N}$, as $\varepsilon \to 0$. 
  Using integration by parts, we obtain 
  \begin{equation}\label{0203-2}
    \int_{0}^{T}\left(\partial_t \left(\theta^{(\varepsilon)}\right)_{\leq N} 
    + \Lambda^{\alpha}\left(\theta^{(\varepsilon)}\right)_{\leq N} 
    + \nabla\cdot\left({\bf u}^{(\varepsilon)}(t)\theta^{(\varepsilon)}(t)\right)_{N}\right)\psi(t) ~{\rm d}t 
    = \theta_{0, N} - \left(\theta^{(\varepsilon)}\right)_{\leq N}(0). 
  \end{equation} 
  Moreover, taking the derivative of \eqref{0203-2} with respect to time, for any $t\in (0,T)$, we have 
  \begin{equation*}
    \partial_t \left(\theta^{(\varepsilon)}\right)_{\leq N} 
      + \Lambda^{\alpha}\left(\theta^{(\varepsilon)}\right)_{\leq N} + \nabla\cdot\left({\bf u}^{(\varepsilon)}\theta^{(\varepsilon)}\right)_{N} =0. 
  \end{equation*}
  Since $\theta_{N}^{(\varepsilon)}$ is smooth, we have 
  \begin{equation*}
    \left(\theta^{(\varepsilon)}\right)_{\leq N}(t) = e^{-t\Lambda^{\alpha}}\left(\theta^{(\varepsilon)}\right)_{\leq N}(0)
                                   - \int_{0}^{T}e^{-(t-s)\Lambda^{\alpha}}\left(\nabla\cdot\left({\bf u}^{(\varepsilon)}\theta^{(\varepsilon)}\right)_{N}\right)~{\rm d}s, 
  \end{equation*}
  and for any $\phi\in \mathcal{S}(\mathbb{R}^2)$, 
  \begin{align*}
    &\subscripts{\mathcal{S}'}{\displaystyle\left\langle\left(\theta^{(\varepsilon)}\right)_{\leq N}(t), \phi \right\rangle}{\mathcal{S}}\\ 
      &=\subscripts{\mathcal{S}'}{\left\langle e^{-t\Lambda^\alpha}\left(\theta^{(\varepsilon)}\right)_{\leq N}(0),\phi\right\rangle}{\mathcal{S}}
      +\subscripts{\mathcal{S}'}{\left\langle \int_{0}^{t}e^{-(t-s)\Lambda^\alpha}\left({\bf u}^{(\varepsilon)}\theta^{(\varepsilon)}\right)_{N} ~{\rm d}s,\nabla\phi \right\rangle}{\mathcal{S}}. 
  \end{align*}
  Therefore, passing to the limit, $\theta$ is a solution of \eqref{SQG} in the sense of Definition~\ref{0215-1}. 
\end{proof}

\vskip 7mm 

\noindent
{\bf Acknowledgments.} The author would like to thank Professor Tsukasa Iwabuchi for valuable discussions and continuous encouragement. 
The author was supported by the Grant-in-Aid for JSPS Fellows, Grant Number 26KJ0527.\\ 
{\bf Data availability statement.} This manuscript has no associated data.\\
{\bf Conflict of Interest.} The author declares that he has no conflict of interest.

\end{document}